\documentclass[11pt]{article}
\usepackage{fullpage, amssymb, amsthm, amsmath}
\usepackage{cancel}
\usepackage{hyperref}
\usepackage{caption}
\newtheorem{thm}{Theorem}
\newtheorem{lem}[thm]{Lemma}
\newtheorem{cor}[thm]{Corollary}

\newtheorem{prop}[thm]{Proposition}

\newtheorem*{mainthm}{Main Theorem}
\newtheorem*{equiv-lem}{Equivalence Lemma}
\theoremstyle{definition}
\newtheorem{example}[thm]{Example}
\newtheorem{defn}[thm]{Definition}
\def\mad{\textrm{mad}}
\def\ch{\textrm{ch}}
\newcommand\vph{\varphi}
\newcommand\bs{\boldsymbol}
\newcommand\flr[1]{\lfloor#1\rfloor}
\newcommand\ceil[1]{\lceil#1\rceil}
\def\aftermath{\par\vspace{-\belowdisplayskip}\vspace{-\parskip}\vspace{-\baselineskip}}

\RequirePackage{marginnote,hyperref}
\newcommand{\aside}[1]{\marginnote{\scriptsize{#1}}[0cm]}
\newcommand{\aaside}[2]{\marginnote{\scriptsize{#1}}[#2]}
\newcommand\Emph[1]{\emph{#1}\aside{#1}}
\newcommand\EmphE[2]{\emph{#1}\aaside{#1}{#2}}

\usepackage{tikz}
\usetikzlibrary{decorations.markings,arrows,snakes,calc}
\usetikzlibrary{positioning}

\tikzstyle{uStyle}=[shape = circle, minimum size = 4pt, inner sep = 1pt,
outer sep = 0pt, fill=white, semithick, draw]
\tikzstyle{u04Style}=[shape = circle, minimum size = 4pt, inner sep = 1pt,
outer sep = 0pt, fill=white, semithick, draw,
label={[yshift=.05cm]above:\footnotesize{$U_0^4$}}]
\tikzstyle{u03Style}=[shape = circle, minimum size = 4pt, inner sep = 1pt,
outer sep = 0pt, fill=white, semithick, draw,
label={[yshift=.05cm]above:\footnotesize{$U_0^3$}}]
\tikzstyle{u03lStyle}=[shape = circle, minimum size = 4pt, inner sep = 1pt,
outer sep = 0pt, fill=white, semithick, draw,
label={[yshift=.05cm,xshift=-.04cm]above:\footnotesize{$U_0^3$}}]
\tikzstyle{u03rStyle}=[shape = circle, minimum size = 4pt, inner sep = 1pt,
outer sep = 0pt, fill=white, semithick, draw,
label={[yshift=.05cm,xshift=.04cm]above:\footnotesize{$U_0^3$}}]
\tikzstyle{u02Style}=[shape = circle, minimum size = 4pt, inner sep = 1pt,
outer sep = 0pt, fill=white, semithick, draw,
label={[yshift=.05cm]above:\footnotesize{$U_{\ell}^2$}}]
\tikzstyle{f02Style}=[shape = circle, minimum size = 4pt, inner sep = 1pt,
outer sep = 0pt, fill=white, semithick, draw,
label={[yshift=.05cm]above:\footnotesize{$F_0^2$}}]
\tikzstyle{f12Style}=[shape = circle, minimum size = 4pt, inner sep = 1pt,
outer sep = 0pt, fill=white, semithick, draw,
label={[yshift=.05cm]above:\footnotesize{$F_1^2$}}]
\tikzstyle{u13Style}=[shape = circle, minimum size = 4pt, inner sep = 1pt,
outer sep = 0pt, fill=white, semithick, draw,
label={[yshift=.05cm]above:\footnotesize{$U_1^3$}}]
\tikzstyle{u12Style}=[shape = circle, minimum size = 4pt, inner sep = 1pt,
outer sep = 0pt, fill=white, semithick, draw,
label={[yshift=.05cm]above:\footnotesize{$U_{\ell+1}^2$}}]
\tikzstyle{u22Style}=[shape = circle, minimum size = 4pt, inner sep = 1pt,
outer sep = 0pt, fill=white, semithick, draw,
label={[yshift=.05cm]above:\footnotesize{$U_{\ell+2}^2$}}]
\tikzstyle{u03LStyle}=[shape = circle, minimum size = 4pt, inner sep = 1pt,
outer sep = 0pt, fill=white, semithick, draw,
label={[yshift=-.05cm]below:\footnotesize{$U_0^3$}}]
\tikzstyle{u03RStyle}=[shape = circle, minimum size = 4pt, inner sep = 1pt,
outer sep = 0pt, fill=white, semithick, draw,
label={[yshift=0cm]left:\footnotesize{$U_0^3$}}]
\tikzstyle{u03RbStyle}=[shape = circle, minimum size = 4pt, inner sep = 1pt,
outer sep = 0pt, fill=white, semithick, draw,
label={[xshift=0.925cm]left:\footnotesize{$U_0^3$}}]
\tikzstyle{u12RStyle}=[shape = circle, minimum size = 4pt, inner sep = 1pt,
outer sep = 0pt, fill=white, semithick, draw,
label={[xshift=1.2cm]left:\footnotesize{$U_{\ell+1}^2$}}]

\tikzstyle{uBstyle}=[shape = circle, minimum size = 9pt, inner sep = 1pt,
outer sep = 0pt, fill=white, semithick, draw]
\tikzstyle{lStyle}=[shape = circle, minimum size = 5pt, inner sep =
0.5pt, outer sep = 0pt, draw=none]
\tikzstyle{vlStyle}=[shape = circle, minimum size = 4pt, inner sep =
1.0pt, outer sep = 0pt, draw, fill=white, semithick, font=\footnotesize]

\tikzstyle{uStyle}=[shape = circle, minimum size = 4pt, inner sep = 1pt,
outer sep = 0pt, fill=white, semithick, draw]
\tikzstyle{IStyle}=[shape = circle, minimum size = 5pt, inner sep = 1pt,
outer sep = 1.5pt, fill=black, semithick, draw]

\title{Vertex Partitions into an Independent Set\\ and a Forest with Each
Component Small}
\author{Daniel W. Cranston\thanks{%
Department of Mathematics and Applied Mathematics, Virginia Commonwealth
University, Richmond, VA, USA;
\texttt{dcranston@vcu.edu}
} 
\and Matthew P. Yancey\thanks{%
Institute for Defense Analyses - Center for Computing Sciences, Bowie, MD, USA; \texttt{mpyancey1@gmail.com}
}
}

\begin{document}
\maketitle
\abstract{For each integer $k\ge 2$, we determine a sharp bound on $\mad(G)$
such that $V(G)$ can be partitioned into sets $I$ and $F_k$, where $I$ is an
independent set and $G[F_k]$ is a forest in which each component has at most $k$
vertices.  For each $k$ we construct an infinite family of examples showing our
result is best possible.  
Our results imply that every planar graph $G$ of girth at least 9 (resp. 8, 7)
has a partition of $V(G)$ into an independent set $I$ and a set $F$ such that
$G[F]$ is a forest with each component of order at most 3 (resp. 4, 6). 

Hendrey, Norin, and Wood asked for the largest function $g(a,b)$
such that if $\mad(G)<g(a,b)$ then $V(G)$ has a partition into sets $A$ and $B$
such that $\mad(G[A])<a$ and $\mad(G[B])<b$.  They specifically asked for the
value of $g(1,b)$, i.e., the case when $A$ is an independent set.
Previously, the only values known were $g(1,4/3)$ and $g(1,2)$.  We find
$g(1,b)$ whenever $4/3< b<2$.
}
\bigskip

\section{Introduction}
An \Emph{$(I,F_k)$-coloring} for a graph $G$ is a
partition of $V(G)$ into sets $I$ and $F$ such that $I$ is an independent set
and $F$ induces a forest in which each component has at most $k$ vertices.
The \emph{average degree of $G$} is $2|E(G)|/|V(G)|$.
The \emph{maximum average degree of $G$}, denoted \Emph{$\mad(G)$}, is the maximum, taken over
all subgraphs $H$, of the average degree of $H$.
In this paper, we prove a sufficient condition for a graph $G$ to have an
$(I,F_k)$-coloring, in terms of $\mad(G)$.

\begin{thm}
\label{mad-thm}
For each integer $k\ge 2$, let 
$$
f(k):=\left\{\begin{array}{ll}\aside{$f(k)$}
3-\frac3{3k-1} & \mbox{$k$ even} \\
3-\frac3{3k-2} & \mbox{$k$ odd} \\
\end{array}
\right.
$$
If $\mad(G)\le f(k)$, then $G$ has an $(I,F_k)$-coloring.
\end{thm}
Theorem~\ref{mad-thm} is best possible.  For each positive integer $k$ there
exists an infinite family of graphs with maximum average degree approaching
$f(k)$ (from above) such that none of these graphs has an $(I,F_k)$-coloring.
Note that $f(3)=\frac{18}7$, $f(4)=\frac{30}{11}$, and $f(6)=\frac{48}{17}$.  
Each planar graph $G$ with girth $g$ has $\mad(G)<\frac{2g}{g-2}$.
So Theorem~\ref{mad-thm} implies that every planar graph $G$ of girth at
least 9 (resp. 8, 7) has a partition of $V(G)$ into an independent set $I$ and
a set $F$ where $G[F]$ is a forest with each tree of order at most 3
(resp. 4, 6); for girth 9, this is best possible, since~\cite[Corollary 4]{EMOP}
constructs girth 9 planar graphs with no $(I,F_2)$-coloring.
This strengthens results in~\cite{DMP,NS-mfmc}.
Theorem~\ref{mad-thm} is implied by a more general result below, 
our Main Theorem.  Before introducing definitions and notation to state it, we
briefly discuss related work.

Choi, Dross, and Ochem \cite{CDO} studied a variant of $(I,F_k)$-colorings where
they did not require the components of $G[F_k]$ to be acyclic, but only to have
order at most $k$.  They proved that $G$ has such a coloring whenever
$\mad(G) < \frac{8}{3}(1-\frac{1}{3k+1})$.  
Theorem~\ref{mad-thm} allows a weaker hypothesis (and a stronger conclusion).
Moreover, the argument on the
sharpness of Theorem~\ref{mad-thm} (see Lemma \ref{sharp-critical-lem})
does not require the acyclic nature of $F_k$, and therefore
Theorem~\ref{mad-thm} is also a sharp result for this variant of the problem.
%
Dross, Montassier, and Pinlou~\cite{DMP} studied a
different variant of $(I,F_k)$-colorings, where $G[F_k]$ has bounded maximum
degree, but perhaps not bounded order (earlier related results are
in~\cite{BK-defective} and~\cite{BIMOR}).  
Under hypotheses very similar to those in Theorem~\ref{mad-thm}, they proved
that $G$ has such a coloring.
These results, too, are strengthened by Theorem~\ref{mad-thm}.

We can also view Theorem~\ref{mad-thm} in a more general context.
Hendrey, Norin, and Wood~\cite[Problem \#14]{HNW} asked for the largest function
\Emph{$g(a,b)$}
such that if $\mad(G)<g(a,b)$ then $V(G)$ has a partition into sets $A$ and $B$
such that $\mad(G[A])<a$ and $\mad(G[B])<b$.  They specifically asked for the
value of $g(1,b)$, which corresponds to the case that $A$ is an independent set.
Nadara and Smulewicz~\cite{NS-mfmc} used maximum flows to give a short
proof that $g(1,b)\ge b+1$ and $g(2,b)\ge b+2$.
However, the only exact values previously known\footnote{Borodin, Kostochka, and
Yancey~\cite{BKY-improper} also showed that $g(4/3,4/3)=14/5$.} were $g(1,4/3)$
and $g(1,2)$ (see~\cite{BK-improper} for $g(1,4/3)$ and see below for $g(1,2)$).
We find the value of $g(1,b)$ whenever $4/3\le b<2$.

We also study a related function \Emph{$\tilde{g}(a,b)$}.  This is
the largest value for which there is a finite set $\mathcal{G}_{a,b}$
of graphs such that if $\mad(G)<\tilde{g}(a,b)$ and $G$ has no graph in
$\mathcal{G}_{a,b}$ as a subgraph, then $V(G)$ has a partition into
sets $A$ and $B$ where $\mad(G[A])<a$ and $\mad(G[B])<b$.  That is,
$\tilde{g}(a,b)$ is the minimum value such that there is an infinite family of graphs
$G_j$ with $\mad(G_j)$ approaching $\tilde{g}(a,b)$ from above
(as $j\to \infty$) and each $V(G_j)$ has no partition $A,B$ with
$\mad(G_j[A])<a$ and $\mad(G_j[B])<b$.  Clearly, $g(a,b)\le \tilde{g}(a,b)$, and
sometimes this inequality is strict.  

In~\cite{CY-IF} we observed that $g(1,2)=3$.  The lower bound follows from
degeneracy.\footnote{Given a vertex $v$ of degree
at most 2, by induction we partition $G-v$ into sets $I$ and $F$ such that $I$
is independent and $G[F]$ is a forest.  If $v$ has no neighbor in $I$, then we
add $v$ to $I$.  Otherwise, we add it to $F$.}  The upper bound $g(1,2)\le 3$
comes from $K_4$.  However, $K_4$ is the single obstruction to strengthening
this bound.  In fact, we proved that $\tilde{g}(1,2)=3.2$.
Each component of a graph $G$ with $\mad(G)<2$ is a forest.
Thus, a partition of $V(G)$ into sets $I$ and $F$ with $\mad(G[I])<1$ and
$\mad(G[F])<2-2/(k+1)$ is precisely an $(I,F_k)$-partition.
In the present paper, we show that $g(1,2-2/(k+1))=\tilde{g}(1,2-2/(k+1))=f(k)$
for every integer $k\ge 2$ (here $f(k)$ is as defined in
Theorem~\ref{mad-thm}).  This is particularly interesting because
$\tilde{g}(1,2)=3.2$, but $\tilde{g}(1,b)<3$ for every $b<2$.

A \Emph{precoloring} of $G$ is a partition of $V(G)$ into sets $U_0, U_1, \ldots,
U_{k-1}, F_1, F_2, \ldots, F_k$\aaside{$U_0,\ldots,U_{k-1}$}{4mm}%
\aaside{$F_1,\ldots,F_k, I$}{8mm}, and $I$.  Intuitively, we think of a vertex in $F_j$ as
being already colored $F$ and having an additional $j-1$ (fake) neighbors that
are also already colored $F$.  So, for example, if a vertex is in $F_k$ then we
cannot color any of its neighbors in $\bigcup_{j=0}^{k-1}U_j$ with $F$, since
this would create a component colored $F$ with at least $k+1$ vertices. 
Similarly, a vertex $v$ in $U_j$ is uncolored, but has $j$ fake neighbors that
are colored $F$. 
So coloring $v$ with $F$ would create a component colored $F$ with $j+1$ vertices.
An $(I,F_k)$-coloring of a precolored graph $G$
is an $(I,F_k)$-coloring $(I',F')$ of the underlying (not precolored) graph $G$
such that $I\subseteq I'$, $\cup_{j=1}^k F_j\subseteq F'$ and each component of
$G[F']$ has at most $k$ vertices \emph{including} any fake neighbors arising from
the precoloring.  A graph $G$ is \emph{precolored trivially} if $U_0=V(G)$, so
$U_1=\cdots=U_{k-1}=F_1=\cdots=F_k=I=\emptyset$.

A precolored graph $G$ is \Emph{$(I,F_k)$-critical} if $G$ has no
$(I,F_k)$-coloring, but every proper subgraph of $G$ does and, furthermore, for
any vertex precolored $U_j$ or $F_j$, reducing $j$ by 1 allows an
$(I,F_k)$-coloring of $G$.  So Theorem~\ref{mad-thm} is equivalent to saying
that every (trivially precolored) $(I,F_k)$-critical graph $G$ has $\mad(G)>f(k)$.
To facilitate a proof by induction, we want to extend Theorem~\ref{mad-thm} to
allow other precolorings.  However, a vertex in $U_j$ (with $j>0$) or in $F_j$ imposes
more constraints on an $(I,F_k)$-coloring than one in $U_0$.  Intuitively, a
vertex in $V(G)\setminus U_0$ should ``count more'' toward the average degree
than one in $U_0$.  This
motivates weighting vertices differently, as we do below.  (In
Section~\ref{gadgets-sec}, we explain our choice of weights.)

\begin{defn}
\label{rho-defn}
\label{coeff-defn}
For each integer $k\ge 2$, let
\begin{itemize}
        \item $C_E:=\{ 3k-1\mbox{ for $k$ even, } 3k-2\mbox{ for $k$
odd}\}$;\aside{$C_E$}
	\item $C_{U,0}:=\frac{3C_E-3}2$;
	\item $C_{U,j} := C_{U,0} - 3j = \frac{3C_E-3}2-3j$ for $0 < j \leq k$;\aside{$C_{U,j}$}
	\item $C_{F,j} := C_{U,j-1} + C_{I} - C_E = C_E - 3j$ for $1 \leq j \leq
\lfloor \frac{k+1}2 \rfloor$;\aside{$C_{F,j}$}
	\item $C_{F,j} := C_{U,\flr{\frac{k-1}2}} + C_{U,\ceil{\frac{k-1}2}} +
C_{U,j-\lfloor \frac{k+3}2 \rfloor} - 3C_E  = 3(k-j)$ for $\lfloor \frac{k+3}2
\rfloor \leq j\le k$; and 
	\item $C_I := C_{U,0} + C_{F,k} - C_E = \frac{C_E - 3}2$.\aside{$C_I$}
\end{itemize}
\end{defn}

\begin{mainthm}
Fix an integer $k\ge 2$.  Let 
$$
\rho_G^k(R):=\sum_{j=0}^{k-1}C_{U,j}|U_j\cap R|+\sum_{j=1}^kC_{F,j}|F_j\cap R|+C_I|I\cap R|-C_E|E(G[R])|,
\aside{$\rho^k_G$}$$
\nopagebreak
for each $R\subseteq V(G)$. If a precolored graph $G$ is $(I,F_k)$-critical, then $\rho_G^k(V(G))\le -3$.
\end{mainthm}
Now is a good time to define more terminology and notation.
We typically write $\rho^k$,
rather than $\rho^k_G$, when there is no danger of confusion.
We also write \Emph{coloring} to mean $(I,F_k)$-coloring.
An \Emph{$F$-component} is a component of $G[F]$ (either for an
$(I,F_k)$-coloring of a graph $G$ or for a precoloring of $G$, where
$F=\cup_{j=1}^kF_j$).
We will often want to move a vertex $v$ from $U_a$ to $U_{a+b}$ or from $F_a$ to
$F_{a+b}$, for some integers $a$ and $b$.  Informally, we call this
``adding $b$ $F$-neighbors to $v$''.  If an uncolored vertex $v$ ever has $k$
or more $F$-neighbors, then we recolor $v$ with $I$ (since coloring $v$ with
$F$ would create an $F$-component with at least $k+1$ vertices, which is forbidden);
see Lemma~\ref{noIFk-lem} and the comment after it.
Note the following easy proposition.

\begin{prop}
The Main Theorem implies Theorem~\ref{mad-thm}.
\end{prop}
\begin{proof}
Observe that $\frac{2C_{U,0}}{C_E}=f(k)$, as defined in Theorem~\ref{mad-thm}.  
Thus, if $G$ is precolored trivially, then the condition $\rho^k(V(G))\ge
0$ is equivalent to $\frac{2|E(G)|}{|V(G)|}\le f(k)$.  By the Main Theorem,
each $(I,F_k)$-critical graph $G$ has $\rho^k(V(G))\le -3$.  Thus, if
$\mad(G)\le f(k)$, then $\rho(R)\ge 0$ for all $R\subseteq V(G)$; so $G$
contains no $(I,F_k)$-critical subgraph.  Hence, $G$ has an $(I,F_k)$-coloring.
\end{proof}

The proof of the Main Theorem differs somewhat depending on whether $k$ is
even or odd.  However, the two cases are similar.  Thus, we begin the
proof (for all $k$) in Section~\ref{proof-start-sec}.  In
Section~\ref{k-even-sec} we conclude it
for $k$ even, and in Section~\ref{k-odd-sec} we conclude it for $k$ odd.
Before proving the Main Theorem, we discuss the sharpness examples and the gadgets
that motivate our weights in Definition~\ref{rho-defn}.  We then conclude the
introduction with a brief overview of the potential method.  

\subsection{Sharpness Examples}
\label{sharpness-sec}

\begin{figure}[!b]
\centering
\begin{tikzpicture}[thick, xscale=.65, yscale=.9]
\tikzset{every node/.style=uStyle}
\tikzset{2-threadAa/.style={postaction=decorate, decoration={markings,
    mark=at position 0.23 with {\node[uStyle] () {};},
    mark=at position 0.78 with {\node[uStyle] () {};} }}}
\tikzset{2-threadAb/.style={postaction=decorate, decoration={markings,
    mark=at position 0.23 with {\node[uStyle] () {};},
    mark=at position 0.77 with {\node[uStyle] () {};} }}}
\tikzset{2-threadB/.style={postaction=decorate, decoration={markings,
    mark=at position 0.27 with {\node[uStyle] () {};},
    mark=at position 0.91 with {\node[uStyle] () {};} }}}
\tikzset{2-threadCa/.style={postaction=decorate, decoration={markings,
    mark=at position 0.33 with {\node[fill=white, draw=none, minimum size=8pt] () {};},
    mark=at position 0.33 with {\node[IStyle] () {};},
    mark=at position 0.68 with {\node[uStyle] () {};} }}}
\tikzset{2-threadCb/.style={postaction=decorate, decoration={markings,
    mark=at position 0.33 with {\node[uStyle] () {};},
    mark=at position 0.61 with {\node[fill=white, draw=none, minimum size=8pt] () {};},
    mark=at position 0.61 with {\node[IStyle] () {};} }}}
\tikzset{2-threadCc/.style={postaction=decorate, decoration={markings,
    mark=at position 0.33 with {\node[uStyle] () {};},
    mark=at position 0.645 with {\node[uStyle] () {};} }}}
\tikzset{2-threadCd/.style={postaction=decorate, decoration={markings,
    mark=at position 0.345 with {\node[fill=white, draw=none, minimum size=8pt] () {};},
    mark=at position 0.345 with {\node[IStyle] () {};},
    mark=at position 0.70 with {\node[uStyle] () {};} }}}
\tikzset{2-threadDa/.style={postaction=decorate, decoration={markings,
    mark=at position 0.42 with {\node[fill=white, draw=none, minimum size=8pt] () {};},
    mark=at position 0.42 with {\node[IStyle] () {};},
    mark=at position 0.85 with {\node[uStyle] () {};} }}}
\tikzset{2-threadDb/.style={postaction=decorate, decoration={markings,
    mark=at position 0.41 with {\node[uStyle] () {};},
    mark=at position 0.78 with {\node[fill=white, draw=none, minimum size=8pt] () {};},
    mark=at position 0.78 with {\node[IStyle] () {};} }}}

\def\radB{1cm}
\def\radS{.65cm}
\def\mythick{.7mm}


\begin{scope}[xshift=.945cm, xscale=.9675]
\draw (300:\radB) node (v0) {} -- (215:.25*\radB) node (w0) {} -- (60:\radB) node (x0) {}
-- (v0);
\draw[line width=\mythick] (x0) -- ++ (30:\radS) node {} -- ++ (150:\radS) node {} -- (x0);
\draw[line width=\mythick] (w0) -- ++ (150:\radS) node {} -- ++ (270:\radS) node {} -- (w0);
\draw[line width=\mythick] (v0) -- ++ (330:\radS) node {} -- ++ (210:\radS) node {} -- (v0);

\begin{scope}[xshift=5cm]
\draw (300:\radB) node (v1) {} -- (215:.25*\radB) node (w1) {} -- (60:\radB) node (x1) {}
-- (v1);
\draw (v0) ++ (-1.2cm,-3mm) node[shape=rectangle, inner sep=0, draw=none]{\tiny{$\flr{(k-2)/2}$}};
\draw (w0) ++ (-1.8cm,0) node[shape=rectangle, inner sep=0, draw=none]{\tiny{$\flr{(k-1)/2}$}};
\draw (x0) ++ (-0.7cm,3mm) node[shape=rectangle, inner sep=0, draw=none]{\tiny{$\flr{k/2}$}};

\end{scope}

\foreach \start/\end\mylabel\where\rot\AB in {
v0/v1/{\flr{(k-2)/2}}/.5/0/Aa,
v0/w1/{\flr{(k-1)/2}}/.6/12/B,
v0/x1/{\flr{k/2}}/.5/20/Ab}
\draw[line width =\mythick, 2-thread\AB] (\start) -- (\end) 
node[pos=\where, shape=rectangle, inner sep=.2, draw=none, rotate=\rot]{\tiny{$\mylabel$}};

\begin{scope}[xshift=10cm]
\draw (300:\radB) node (v2) {} -- (215:.25*\radB) node (w2) {} -- (60:\radB) node (x2) {}
-- (v2);
\end{scope}

\foreach \start/\end\mylabel\where\rot\AB in {
v1/v2/{\flr{(k-2)/2}}/.5/0/Aa,
v1/w2/{\flr{(k-1)/2}}/.6/12/B,
v1/x2/{\flr{k/2}}/.5/20/Ab}
\draw[line width =\mythick, 2-thread\AB] (\start) -- (\end) 
node[pos=\where, shape=rectangle, inner sep=.2, draw=none, rotate=\rot]{\tiny{$\mylabel$}};

\begin{scope}[xshift=15cm]
\draw (300:\radB) node (v3) {} -- (215:.25*\radB) node (w3) {} -- (60:\radB) node (x3) {}
-- (v3);
\end{scope}

\foreach \start/\end\mylabel\where\rot\AB in {
v2/v3/{\flr{(k-2)/2}}/.5/0/Aa,
v2/w3/{\flr{(k-1)/2}}/.6/12/B,
v2/x3/{\flr{k/2}}/.5/20/Ab}
\draw[line width =\mythick, 2-thread\AB] (\start) -- (\end) 
node[pos=\where, shape=rectangle, inner sep=.2, draw=none, rotate=\rot]{\tiny{$\mylabel$}};

\draw (v3) -- ++ (330:\radS) node {} -- ++ (210:\radS) node {} -- (v3);
\end{scope}

%
%

\begin{scope}[yshift=-4cm, xscale=.7]


\begin{scope}[xshift=-2cm]
\draw (300:\radB) node (v0) {} (230:.25*\radB) node (w0) {} (60:\radB)
node[IStyle] (x0) {} (v0) -- (w0) -- (x0) -- (v0);

\draw[line width=\mythick] (x0) -- ++ (30:\radS) node {} -- ++ (150:\radS) node {} -- (x0);
\draw[line width=\mythick] (w0) ++ (150:\radS) node (w0a) {} ++ (270:\radS)
node[IStyle] (w0b) {} (w0a) -- (w0) -- (w0b) -- (w0a);
\draw[line width=\mythick] (v0) -- ++ (330:\radS) node (v0a) {} ++ (210:\radS)
node[IStyle] (v0b) {} (v0) -- (v0a) -- (v0b) -- (v0);

\begin{scope}[xshift=3.5cm]
\draw (300:\radB) node (v1) {} (230:.25*\radB) node (w1) {} (60:\radB)
node[IStyle] (x1) {} (v1) -- (w1) -- (x1) -- (v1);
\end{scope}

\foreach \start/\end\AB in {v0/v1/Ca, v0/w1/Da, v0/x1/Cd}
\draw[line width =\mythick, 2-thread\AB] (\start) -- (\end) node {};

\begin{scope}[xshift=7cm]
\draw (300:\radB) node (v2) {} (230:.25*\radB) node (w2) {} (60:\radB)
node[IStyle] (x2) {} (v2) -- (w2) -- (x2) -- (v2);
\end{scope}

\foreach \start/\end\AB in {v1/v2/Ca, v1/w2/Da, v1/x2/Cd}
\draw[line width =\mythick, 2-thread\AB] (\start) -- (\end);

\begin{scope}[xshift=10.5cm]
\draw (300:\radB) node (v3) {} (230:.25*\radB) node (w3) {} (60:\radB)
node[IStyle] (x3) {} (v3) -- (w3) -- (x3) -- (v3);
\end{scope}

\foreach \start/\end\AB in {v2/v3/Ca, v2/w3/Da, v2/x3/Cd}
\draw[line width =\mythick, 2-thread\AB] (\start) -- (\end);

\draw (v3)  ++ (330:\radS) node (v3a) {} ++ (210:\radS) node[IStyle] (v3b) {}
(v3a) -- (v3b) -- (v3);

\begin{scope}[xshift=3.5cm] 
\draw (60:\radB) node[IStyle] {};
\end{scope}

\end{scope}


\begin{scope}[xshift=13cm]

\draw (300:\radB) node[IStyle] (v0) {} (230:.25*\radB) node (w0) {} (60:\radB)
node (x0) {} (v0) -- (w0) (x0) -- (v0);

\draw[line width=\mythick] (x0) ++ (30:\radS) node[IStyle] (x0a) {} (x0) --
(x0a) -- ++ (150:\radS) node {}  -- (x0);
\draw[line width=\mythick] (w0) ++ (150:\radS) node (w0a) {} ++ (270:\radS)
node[IStyle] (w0b) {} (w0a) -- (w0) -- (w0b) -- (w0a);
\draw[line width=\mythick] (v0) -- ++ (330:\radS) node (v0a) {} ++ (210:\radS)
node (v0b) {} (v0) -- (v0a) -- (v0b) -- (v0);

\begin{scope}[xshift=3.5cm]
\draw (300:\radB) node[IStyle] (v1) {} (230:.25*\radB) node (w1) {} (60:\radB)
node (x1) {} (v1) -- (w1) -- (x1) -- (v1);
\end{scope}

\foreach \start/\end\AB in {v0/v1/Cc, v0/w1/Db, v0/x1/Cb}
\draw[line width =\mythick, 2-thread\AB] (\start) -- (\end) node {};

\begin{scope}[xshift=7cm]
\draw (300:\radB) node[IStyle] (v2) {} (230:.25*\radB) node (w2) {} (60:\radB)
node (x2) {} (v2) -- (w2) -- (x2) -- (v2);
\end{scope}

\foreach \start/\end\AB in {v1/v2/Cc, v1/w2/Db, v1/x2/Cb}
\draw[line width =\mythick, 2-thread\AB] (\start) -- (\end);

\begin{scope}[xshift=10.5cm]
\draw (300:\radB) node[IStyle] (v3) {} (230:.25*\radB) node (w3) {} (60:\radB)
node (x3) {} (v3) -- (w3) -- (x3) -- (v3);
\end{scope}

\foreach \start/\end\AB in {v2/v3/Cc, v2/w3/Db, v2/x3/Cb}
\draw[line width =\mythick, 2-thread\AB] (\start) -- (\end);

\draw (v3) -- ++ (330:\radS) node {} -- ++ (210:\radS) node {} -- (v3);

\begin{scope}[xshift=3.5cm] 
\draw (300:\radB) node[IStyle] {};
\end{scope}

\end{scope} 
\end{scope} 

\end{tikzpicture}
\caption{Top: The sharpness example $G_{k,3}$.  
Bold edges denote multiple pendent 3-cycles at a vertex or multiple 2-threads
between two vertices.
Bottom left: 
An $(I,F_k)$-coloring of $G_{k,3}-e$, where $e$ is on the 3-cycle pendent at $v_t$.
Bottom right: 
An $(I,F_k)$-coloring of $G_{k,3}-w_0x_0$. 
(Throughout, vertices in $I$ are black and vertices in $F$ are white.)
\label{sharpness-fig}
}
\end{figure}
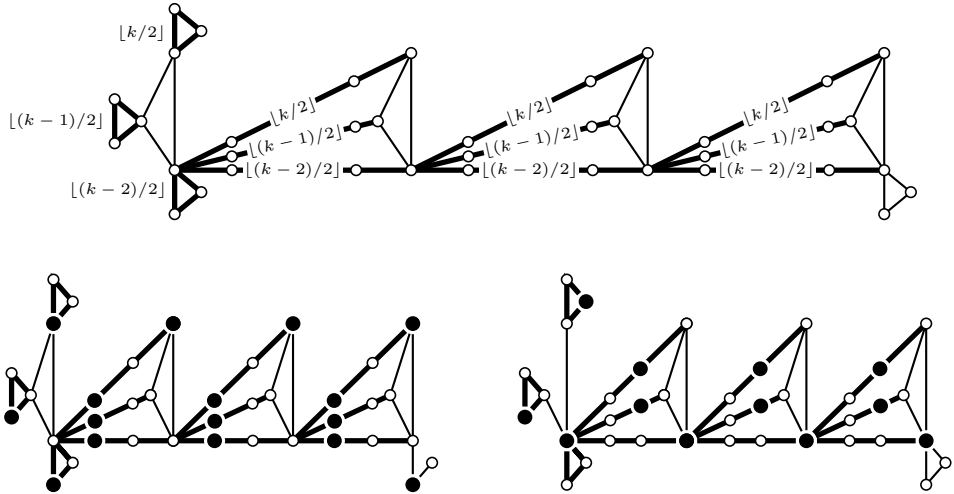

\begin{example}
We write \emph{add a pendent 3-cycle at a vertex $z$} to mean identify $z$
with a vertex of a new 3-cycle.  \emph{Adding $\ell$ pendent 3-cycles at $z$}
means repeating this $\ell$ times.  Similarly, \emph{adding a 2-thread from
$y$ to $z$} means adding new vertices $y'$ and $z'$ and new edges
$yy',y'z',z'z$.  (Adding $\ell$ 2-threads is defined analogously.)

We form an $(I,F_k)$-critical graph \Emph{$G_{k,t}$} as follows
(Figure~\ref{sharpness-fig} shows $G_{k,3}$).  Start with vertices
$v_0,\ldots,v_t$, $w_0,\ldots,w_t$, $x_0,\ldots,x_t$, where $\{v_j,w_j,x_j\}$ induces
$K_3$ for each $j\in\{0,\ldots,t\}$.  Now add $\flr{\frac{k-2}2}$ pendent 3-cycles
at $v_0$, $\flr{\frac{k-1}2}$ pendent 3-cycles at $w_0$, and $\flr{\frac{k}2}$ pendent
3-cycles at $x_0$.  For each $j\in\{1,\ldots,t\}$, add $\flr{\frac{k-2}2}$ 2-threads
from $v_{j-1}$ to $v_j$, $\flr{\frac{k-1}2}$ 2-threads from $v_{j-1}$ to
$w_j$, and $\flr{\frac{k}2}$ 2-threads from $v_{j-1}$ to $x_j$.  Finally, add a
single pendent 3-cycle at $v_t$.
\end{example}

The proof that $G_{k,t}$ is $(I,F_k)$-critical is a bit tedious, but we include
it below for completeness.  It is not needed for the proof of our Main Theorem,
so the reader should feel free to skim (or skip) it.
Intuitively, if we start to color $G_{k,t}$ from the left, each $v_j$ will be in
an $F$-component of order $k$; but for $v_t$, due to the extra pendent 3-cycle,
we get an $F$-component of order $k+1$, a contradiction.  When we delete some
edge $e$, at some point we are able to use $I$ on some $v_{j'}$, and we
continue using $I$ on each $v_j$ with $j\ge j'$.  The coloring of $G_{k,t}-e$
is some combination of the two colorings at the bottom of
Figure~\ref{sharpness-fig}.  (It is interesting to note that the
family $G_{2,t}$ is precisely those sharpness examples given by Borodin and
Kostochka in~\cite{BK-improper}.)

\begin{lem}
$G_{k,t}$ is $(I,F_k)$-critical for all integers $k\ge 2$ and $t\ge 0$.
\label{sharp-critical-lem}
\end{lem}
\begin{proof}
Let $G^j_{k,t}$ denote the subgraph of $G_{k,t}$ induced by
$v_0,\ldots,v_j,w_0,\ldots,w_j,x_0,\ldots,x_j$ along with their pendent
3-cycles and any 2-threads between them.
We show by induction that $G^j_{k,t}$ has an $(I,F_k)$-coloring for each $j<t$;
furthermore, in each such coloring $v_j$ is in an $F$-component of order $k$.
Consider $G^0_{k,t}$.  Because of their pendent 3-cycles, $w_0$ and $x_0$ will
have at least $\flr{\frac{k-1}2}$ and $\flr{\frac{k}2}$ $F$-neighbors (respectively) in
every coloring of $G^0_{k,t}$.  If both $w_0$ and $x_0$ are colored $F$, then
they lie in an $F$-component of order at least $\flr{\frac{k-1}2} +
\flr{\frac{k}2}+2 = k+1$, a contradiction.  So one of $w_0$ and $x_0$ must
be colored $I$.  Thus, $v_0$ is colored $F$; so $v_0$ lies in an $F$-component
of order at least $\flr{\frac{k-2}2}+\flr{\frac{k-1}2}+2=k$.  To see that
$G^0_{k,t}$ has a coloring, color $x_0$ with $I$ and $v_0$ and $w_0$ with $F$.
For each 3-cycle pendent at $v_0$ or $w_0$, use $I$ on one vertex and $F$ on
the other.  For each 3-cycle pendent at $x_0$, use $F$ on both vertices.  This
proves the base case.

Now we consider the induction step.  Since $v_{j-1}$ is in an $F$-component of
order $k$ in $G^{j-1}_{k,t}$, each neighbor of $v_{j-1}$ on a 2-thread
to $\{v_j,w_j,x_j\}$ must be colored $I$; thus, each of \emph{their} neighbors
must be colored $F$.  Now the analysis is nearly identical that that for $j=0$.
To extend the coloring to all of $G^j_{k,t}$, color $x_j$ with $I$ and color
$v_j$ and $w_j$ with $F$.  If we instead tried to color $v_j$ with $I$, then
$w_j$ and $x_j$ must both be colored $F$, so they lie in an $F$-component of
order $\flr{\frac{k}2}+\flr{\frac{k-1}2}+2=k+1$, a contradiction.

To see that $G_{k,t}$ has no coloring, note that such a coloring would
have $v_t$ in an $F$-component of order $k$ (as in the induction step
above).  However, due to the extra pendent 3-cycle at $v_t$, this creates an
$F$-component of order $k+1$, a contradiction.

Finally, we show that $G_{k,t}$ is $(I,F_k)$-critical.  That is, for each $e\in
E(G_{k,t})$ subgraph $G_{k,t}-e$ has a coloring.  By induction we
first prove the stronger statement that if $e\in E(G^{j-1}_{k,t})$, then
$G^j_{k,t}-e$ has a coloring with $v_j$ colored $I$.  (The intuition
is that once we get this for some $j'$, then we can ensure it for all $j'>j$,
so can finish the coloring.) Afterward, we use this to prove that $G_{k,t}-e$
has an $(I,F_k)$-coloring for every $e\in E(G_{k,t})$.

Base case: $j=1$.  If $e$ is not on a pendent 3-cycle at $v_0$, then
$G^{j-1}_{k,t}-e$ has a coloring  in which $v_0$ is colored $I$, as follows.  Either (a)
$e\in \{v_0w_0,v_0x_0\}$, so we can color two vertices in $\{v_0,w_0,x_0\}$ with
$I$ or (b) $e=w_0x_0$ or $e$ is on a 3-cycle pendent at $w_0$ or $x_0$, so we
can color both $w_0$ and $x_0$ with $F$.  If we can color $v_0$ with $I$, then
we extend to $G^1_{k,t}-e$ by using $F$ on all neighbors of $v_0$ on 2-threads,
using $I$ on $v_1$ and neighbors of $w_1$ and $x_1$ on 2-threads, and using $F$
on all remaining vertices.
Assume instead that $e$ is on a pendent 3-cycle at $v_0$.  Now we color both
endpoints of $e$ with $I$, so that $v_0$ is in an $F$-component of order only
$k-1$.  This enables us to use $F$ on some neighbor of $v_0$ on a 2-thread to
$x_1$ (and use $I$ on its neighbor on that 2-thread).  Now we use $F$ on $w_1$
and $x_1$, and use $I$ on $v_1$.  This finishes the base case.

The induction step is nearly identical to the base case.  Suppose $e\in
E(G^{j-1}_{k,t})$.  If $e\in E(G^{j-2}_{k,t})$, then $G^{j-1}_{k,t}-e$ has a
coloring in which $v_{j-1}$ uses $I$.  We extend it to $G^j_{k,t}-e$ in exactly
the same way as extending the coloring of $G^0_{k,t}-e$ to $G^1_{k,t}-e$ above.
Otherwise $e\in E(G^{j-1}_{k,t})\setminus E(G^{j-2}_{k,t})$. Recall, from above,
that $G^{j-2}_{k,t}$ has a coloring, and it has $v_{j-2}$ in an $F$-component of
order $k$.  Now the extension to $G^{j-1}_{k,t}$ is nearly identical to coloring
$G^0_{k,t}-e$ (from the base case at the start of the proof).  This proves our
stronger statement by induction.

Finally, we prove that $G_{k,t}-e$ has a coloring for every $e\in E(G_{k,t})$.
If $e$ is not on the 3-cycle pendent at $v_t$, then we can color $G_{k,t}-e$
with $I$ on $v_t$, so the extra pendent 3-cycle does not matter.  If $e$ is on
the pendent 3-cycle, then we color so that $v_t$ is in an $F$-component of order
$k$ without the extra 3-cycle.  However, now $v_t$ has only a single neighbor on
that pendent 3-cycle, so we color that neighbor with $I$ and the remaining
vertex with $F$.
\end{proof}

\subsection{Gadgets: Where the Coefficients Come From}
\label{gadgets-sec}
Here we explain our choice of weights in Definition~\ref{rho-defn}:  $C_E$,
$C_{U,j}$, $C_{F,j}$, $C_I$.  Everything starts with our sharpness examples in
Section~\ref{sharpness-sec}.  We must choose $C_{U,0}$ and $C_E$ so that all of
these examples have the same potential, i.e., $\rho^k(G_{k,t+1})=\rho^k(G_{k,t})$
for all positive $t$.  Note that
$|V(G_{k,t+1})|-|V(G_{k,t})|=3+2(\flr{\frac{k}2}+\flr{\frac{k-1}2}+\flr{\frac{k-2}2})=C_E$ and
$|E(G_{k,t+1})|-|E(G_{k,t})|=3+3(\flr{\frac{k}2}+\flr{\frac{k-1}2}+\flr{\frac{k-2}2})=C_{U,0}$.
\emph{This} is how we chose $C_E$ and $C_{U,0}$.

For each of $I$, $F_j$ and $U_j$ ($j>0$) we construct a gadget, consisting of
edges and vertices in $U_0$.  Each gadget has a specified vertex $v$ which the
gadget simulates having the desired precoloring; see Figure~\ref{gadgets-fig}.  
The easiest of these is $U_1$.
The gadget is simply a 3-cycle.  Suppose we add a pendent 3-cycle $C$ at any
vertex $v$.  In any coloring of $G$ (with $C$ added), at
least one neighbor of $v$ on $C$ is colored $F$.  Further, if $v$ is colored
$F$, then we can color the remaining vertices of $C$ so that exactly one is in $F$.
Thus, this gadget precisely simulates $v$ being in $U_1$.  For each larger $j$,
the gadget for $U_j$ simply adds $j$ pendent 3-cycles at $v$.  Alternatively, we
can define the gadgets recursively, where adding a pendent 3-cycle moves a
vertex from $U_j$ to $U_{j+1}$.

\begin{figure}[!th]
\centering
\def\radB{1cm}
\def\radS{.65cm}
\def\mythick{.7mm}
\begin{tikzpicture}[semithick]
\tikzset{every node/.style=uStyle}

\begin{scope}[xshift=1cm]
\draw (270:\radB) node (v0) {} (150:\radB) node (w0) {} (30:\radB) node (x0) {}
(v0) -- (w0) -- (x0) -- (v0);
\draw (v0) ++ (.35cm,0) node[draw=none] {\footnotesize{$v$}};
\draw (v0) ++ (0,-.5cm) node[rectangle, draw=none] {\scriptsize{$U_j\to
U_{j+1}$ (always)}};
\draw (v0) ++ (0,-.9cm) node[rectangle, draw=none] {\scriptsize{$F_j\to
F_{j+1}$ ($j\ne \flr{(k+1)/2})$}};
\end{scope}

\begin{scope}[xshift=5cm]
\draw (270:\radB) node (v0) {} (150:\radB) node (w0) {} (30:\radB) node (x0) {}
(v0) -- (w0) -- (x0) -- (v0);
\draw (v0) ++ (.35cm,0) node[draw=none] {\footnotesize{$v$}};
\draw (v0) ++ (0,-.5cm) node[rectangle, draw=none] {\footnotesize{$U_0\to
F_{\flr{(k+3)/2}}$}};
\draw[line width=\mythick] (x0) -- ++ (60:\radS) node {} -- ++ (180:\radS) node {} -- (x0);
\draw[line width=\mythick] (w0) -- ++ (60:\radS) node {} -- ++ (180:\radS) node {} -- (w0);
\draw (w0) ++ (0,.9cm) node[draw=none,rectangle] {\scriptsize{$\flr{(k-1)/2}$}};
\draw (x0) ++ (0,.9cm) node[draw=none,rectangle] {\scriptsize{$\flr{k/2}$}};
\end{scope}

\begin{scope}[xshift=8cm]
\draw (270:\radB) node (v0) {} --++ (0,1.5cm) node (w) {};
\draw (v0) ++ (.35cm,0) node[draw=none] {\footnotesize{$v$}};
\draw (w) ++ (.4cm,0) node[draw=none] {\footnotesize{$F_k$}};
\draw (v0) ++ (0,-.5cm) node[rectangle, draw=none] {\footnotesize{$U_0\to I$~~~}};
\end{scope}

\begin{scope}[xshift=11cm]
\draw (270:\radB) node (v0) {} --++ (0,1.5cm) node (w) {};
\draw (v0) ++ (.35cm,0) node[draw=none] {\footnotesize{$v$}};
\draw (w) ++ (.35cm,0) node[draw=none] {\footnotesize{$I$}};
\draw (v0) ++ (0,-.5cm) node[rectangle, draw=none] {\footnotesize{$U_0\to F_1$}};
\end{scope}

\end{tikzpicture}
\caption{Gadgets to simulate precoloring.\label{gadgets-fig}}
\end{figure}
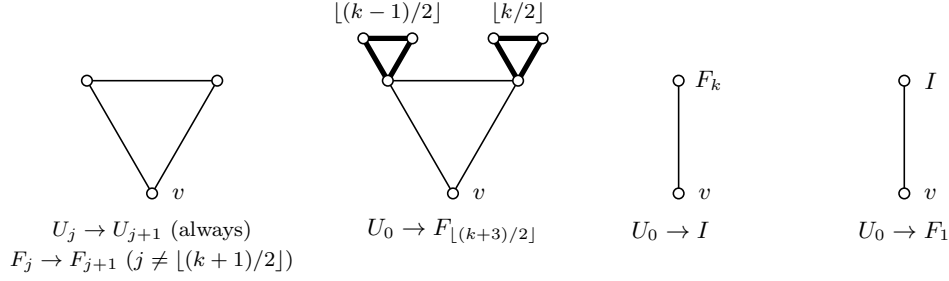

But how do we simulate a vertex in $F_1$?  It is simpler (surprisingly) to start
with the gadget for $F_k$.  This is just the subgraph of $G_{k,t}$ induced by
$v_0,w_0,x_0$ and their pendent 3-cycles. 
Precisely, it is formed from a $K_3$ by adding $\flr{\frac{k-2}2}$ pendent 3-cycles
at $v$ and adding $\flr{\frac{k-1}2}$ and $\flr{\frac{k}2}$ pendent 3-cycles at the two other
vertices of the $K_3$; see the left end of Figure~\ref{sharpness-fig}.  In the
proof of Lemma~\ref{sharp-critical-lem}, we
showed that any coloring of this subgraph must have $v_0$ in an $F$-component
of order $k$.  The potential of this subgraph is 0, so $C_{F,k}=0$.  The gadget
for $I$ is simply an edge to a vertex in $F_k$.  So $C_I = C_{U,0}-C_E+C_{F,k}
= \frac{C_E-3}2$.  Finally, the gadget for $F_1$ is an edge to a vertex in $I$. 
So $C_{F,1}=C_{U,0}+C_I-C_E=C_E-3$.  Adding a pendent 3-cycle at a vertex in
$C_{F,j}$ moves it to $C_{F,j+1}$.  So we are tempted to say that
$C_{F,j+1}=C_{F,j}-3$ for all $j$; but this is not quite right!  We must
simulate each $F_j$ as efficiently as possible.  We can do slightly better for
$F_{j'}$ when $j'=\flr{\frac{k+3}2}$.  The best gadget for $F_{j'}$ is shown in
Figure~\ref{gadgets-fig}; it is formed from the gadget for $F_k$ by
\emph{removing} $k-j'$ pendent 3-cycles at $v_0$.  This gadget
gives $C_{F,j'}=3k-3j'$ (rather than $C_E-3j'$, which we get if we build up
from the gadget for $F_1$).  Now for each $j>j'$, we add $j-j'$ pendent
3-cycles at $v$.  Thus, $C_{F,j}=3k-3j$ for all $j\ge j'$.
\smallskip

It is enlightening to notice that the Main Theorem is logically equivalent to its
restriction to graphs that are precolored trivially.  Since this is not needed for our
proof of the Main Theorem, we are content to provide only a proof sketch.

\begin{equiv-lem}
The Main Theorem is true if and only if it is true when restricted to graphs
with no precolored vertices.
\end{equiv-lem}
\begin{proof}[Proof Sketch]
The case with a trivial precoloring is clearly implied by the general case.
Now we show the reverse implication.
Suppose the Main Theorem is false for some specific value of $k$.  Let $G$ be a
counterexample; among all counterexamples, choose one that minimizes $|V(G)|$.
We will construct another counterexample $\widehat{G}$ (for the same value of $k$)
with $U_0=V(\widehat{G})$.

If $G$ has a vertex $v$ precolored $I$, then we form $G'$ from $G-v$ 
by coloring each neighbor of $v$ (in $G$) with $F$.  Since $G$ is
$(I,F_k)$-critical, so is $G'$.  Since $G'$ is smaller than $G$, we know that
$\rho^k_{G'}(V(G'))\le -3$.  It is straightforward to check that $\rho^k_G(V(G))\le
\rho^k_{G'}(V(G'))\le -3$ (see Lemma~\ref{noIFk-lem} for details); so $G$ is not
a counterexample, a contradiction.  Thus, $I=\emptyset$. 

Now we form a graph $\widehat{G}$ from $G$ by identifying each vertex $v\in V(G)$
colored $U_j$ or $F_j$ with the vertex $v$ in the corresponding gadget (and
removing the precoloring).  It is easy to check that $-2\le \rho^k_G(V(G)) =
\rho^k_G(V(\widehat{G}))$; indeed, this is exactly why we chose the values we did
for $C_{U,j}$ and $C_{F,j}$.  So all that remains is to show that $\widehat{G}$
is $(I,F_k)$-critical.  

First, note that each
gadget precisely simulates the precoloring. That is, every $(I,F_k)$-coloring of
the gadget for each $U_j$ either gives $v$ at least $j$ $F$-neighbors or it colors $v$
with $I$; furthermore, some coloring of the gadget for $U_j$ colors $v$ with $I$
and some other coloring of the gadget for $U_j$  colors $v$ with $F$ and gives
$v$ exactly $j$ $F$-neighbors.  Similarly, every $(I,F_k)$-coloring
of the gadget for each $F_j$ colors
$v$ with $F$ and puts it in an $F$-component of order at least $j$; and some
coloring of the gadget for $F_j$ colors $v$ with $F$ and puts it in a component
of order exactly $j$.  Second, note that deleting any edge from the gadget for
$U_j$ allows a coloring in which $v$ has at most $j-1$ $F$-neighbors. 
Similarly, deleting any edge from the gadget for $F_j$ allows a coloring in
which $v$ is in an $F$-component of order at most $j-1$.  Thus, $\widehat{G}$
is $(I,F_k)$-critical.
\end{proof}

Since the Main Theorem is equivalent to its restriction to graphs with trivial
precolorings, what is the point of allowing precolorings?  The point is to order
the graphs in a way that is more useful for induction (note that
$V(\widehat{G})>V(G)$, so allowing precolorings enables us to simulate
$\widehat{G}$ with a precolored graph $G$ that is smaller than $\widehat{G}$). 
In fact, we could write the whole proof without precolorings, but the partial
order on the graphs needed for that version would be much harder to understand
and keep track of.

\subsection{The Potential Method: A Brief Introduction}
\label{potential-sec}
The function $\rho^k$ is called the \Emph{potential function}, and the
proof technique we employ in this paper is called the \EmphE{potential method}{3mm}. 
Here we give a brief overview.

The essential first step in any proof using the potential method
is to find an infinite family of sharpness examples.  These examples determine 
a sharp necessary condition on $\mad(G)$.  So we use them to choose the 
coefficients $C_{U,0}$ and $C_E$, which define $\rho$.  The necessary
generalization (allowing precoloring and specifically all the different options
$U_j$ and $F_j$) varies with the problem.  For some problems, we do not use precoloring
at all.  In one case we allowed parallel edges~\cite{CY-IF}.  Whenever a
generalization allows precolorings, the coefficients are all determined by the
gadgets, as discussed in the previous section (so it is essential to find the
gadgets with highest potential).

Behind every proof using the potential method is a typical proof using
reducibility and discharging.  Consider, for example, Theorem~\ref{mad-thm}.
Suppose we are aiming to prove that theorem and we want to show that a certain
configuration $H$ is reducible.  Typically, we color $G-V(H)$ by induction and
then show how to extend the coloring to $V(H)$.  The reason we can color
$G-V(H)$ by induction is that, by definition, $\mad(G-V(H))\le \mad(G)$; since
$G-V(H)$ is smaller than $G$, the theorem holds for $G-V(H)$.  \emph{The heart of the
potential method is to show that we can slightly modify $G-V(H)$ before we color
it by induction.}  This modification (say, adding some $F$-neighbors) enables us
to require more of our coloring of $G-V(H)$.  Since this coloring of $G-V(H)$ 
is more constrained, we may be able to extend it to $V(H)$, even if we could not
do so for an arbitrary $(I,F_k)$-coloring of $G-V(H)$.  To make all of this
precise, we need a lower bound on $\rho^k(R)$ for all $R\subsetneq V(G)$.
Such a bound is called a Gap Lemma.  Our modifications may lower $\rho^k(R)$,
but if we can ensure that even this lowered potential is at least $-2$ for all
$R$, then we know by induction that $G'$ cannot contain an $(I,F_k)$-critical
subgraph, so it must have an $(I,F_k)$-coloring.

Once we have proved that various configurations are reducible, we use
discharging to show that a (hypothetical, smallest) counterexample $G$ to our
Main Theorem cannot exist.  We assign charge so that the 
assumption $\rho^k(V(G))\ge -2$ implies that the sum of
all initial charges is at most 4. (This is analogous,
for graphs with $\mad<\alpha$, to using the initial charge
$\ch(v):=d(v)-\alpha$.)  As a first step, we show that each vertex ends with
nonnegative charge.  With a bit more work, we show that if $G$ has no
coloring, then its total charge exceeds 4, so $G$ is not a counterexample.

Our proof of the Main Theorem naturally translates into a polynomial-time
algorithm.  This is typical of proofs using the potential method.  The
translation is mostly straightforward.  The least obvious step is efficiently
finding a set of minimum potential, which can be done using a max-flow/min-cut
algorithm.  We discuss algorithms at length in~\cite[Sections~2.3 and 5]{CY-IF}.

\section{Starting the Proof of the Main Theorem}
\label{proof-start-sec}
Fix an integer $k\ge 2$.  In what follows, we typically write $\rho$ rather than $\rho^k$.  
We say that a graph $G_1$ is \Emph{smaller} than a graph $G_2$ if either (a)
$|V(G_1)|<|V(G_2)|$ or (b) $|V(G_1)|=|V(G_2)|$ and $|E(G_1)|<|E(G_2)|$.
Assume that the Main Theorem is false for $k$.  Let $G$ be a smallest counterexample.
In this section, we prove a number of lemmas restricting the structure of $G$.

\begin{lem}
$I\cup 
U_k \cup
F_k=\emptyset$.
\label{noI-lem}
\label{noFk-lem}
\label{noUk-lem}
\label{noIFk-lem}
\end{lem}
\begin{proof}
Assume, to the contrary, that $I\cup U_k\cup F_k\ne\emptyset$.
First, suppose there exists $v\in F_k$.  Form $G'$ from $G$ by
deleting $v$ and adding each neighbor of $v$ to $I$.  For each $R'\subseteq
V(G')$, subgraph $G'[R']$ has an $(I,F_k)$-coloring if and only if
$G[R'\cup\{v\}]$ does.  Since $G$ is $(I,F_k)$-critical, so is $G'$.  Since
$G'$ is smaller than $G$, by the minimality of $G$, we have $\rho_{G'}(V(G'))\le
-3$.  However, now $\rho_G(V(G)) \le \rho_{G'}(V(G'))+(C_{U,0}-C_I-C_E)d(v) =
\rho_{G'}(V(G'))\le-3$.  Thus, $G$ is not a counterexample.

Suppose instead there exists $v\in I$.  Form $G'$ from $G$ by deleting $v$ and adding
each neighbor of $v$ to $F$ (we assume $d(v)\ge 1$).  
For each $R'\subseteq V(G')$, subgraph $G'[R']$ has an $(I,F_k)$-coloring if and
only if $G[R'\cup\{v\}]$ does.  Since $G$ is $(I,F_k)$-critical, so is $G'$.
Since $G'$ is smaller than $G$, by the minimality of $G$ we have
$\rho_{G'}(V(G'))\le -3$.  
Coloring a vertex in $U_j$ with $F$ moves it to $F_{j+1}$, so decreases its
potential by $C_{U,j}-C_{F,j+1}\le \frac{3C_E-3}2-3j-(C_E-3(j+1)) = \frac{C_E+3}2$.
So $\rho_G(V(G)) \le \rho_{G'}(V(G'))+(\frac{C_E+3}2)d_G(v)-C_Ed_G(v)+C_I
=\rho_{G'}(V(G'))+(\frac{3-C_E}2)d(v)+\frac{C_E-3}2\le \rho_{G'}(V(G'))\le -3$.  Thus,
$G$ is not a counterexample.

Finally, suppose there exists $v\in U_k$.  Form $G'$ from $G$ by coloring $v$
with $I$.  For each $R'\subseteq V(G')$, subgraph $G'[R']$ has an
$(I,F_k)$-coloring if and only $G[R']$ does.  Since $G$ is $(I,F_k)$-critical,
so is $G'$.  Note that $\rho_{G'}(V(G'))=\rho_G(V(G))-C_{U,k}+C_I >
\rho_G(V(G))$.  Now repeating the argument in the previous paragraph 
shows that $G$ is not a smallest counterexample.
\end{proof}

At various points in our proof, we will construct a graph $G'$ from some
subgraph of $G$ by adding $F$-neighbors to one or more vertices.  If this ever
produces an uncolored vertex $v$ with at least $k$ $F$-neighbors, then we
recolor $v$ with $I$, as in the final paragraph of the previous proof.

\begin{lem}
\label{no two adjacent F}
For each edge $vw$, at least one of $v$ and $w$ is in $U$.
\label{noFedge-lem}
\end{lem}
\begin{proof}
Suppose, to the contrary, that $v \in F_i$ and $w \in F_j$.  Form $G'$ from $G$ by contracting edge
$vw$ to create a new vertex $v*w \in F_{i+j}$.  Further, for each vertex $x$
incident to both $v$ and $w$, remove edges $vx$ and $wx$ and put $x$ into $I$.
Contracting edge $vw$ decreases potential by $(C_{F,i}+C_{F,j}-C_E)-C_{F,i+j}\le 0$.
Putting a vertex $x$ into $I$ and deleting two incident edges decreases
potential by at most $C_{U,0}-2C_E-C_I=-C_E$; that is, it increases potential by
at least $C_E$.  Since $G'$ is smaller than $G$, we have $\rho_{G'}(V(G'))\le
-3$.  Thus, $\rho_{G}(V(G))\le \rho_{G'}(V(G'))\le -3$.  So $G$ is not a
counterexample.
\end{proof}

\begin{lem}
\label{delta2-lem}
For each $v\in V(G)$, either $d(v)\ge 2$ or $v\in F_j$ with $j\ge \flr{\frac{k+3}2}$.
\end{lem}
\begin{proof}
Assume, to the contrary, that $d(v)\le 1$ and $v\notin F_j$ with $j\ge
\flr{\frac{k+3}2}$.  Since $G$ is critical, it is connected, so $d(v)=1$; denote
the unique neighbor of $v$ by $w$.  If $v$ is uncolored, then color $G-v$ by
the minimality of $G$.  Now extend this coloring to $G$ by coloring $v$ with
the color not used on $w$.  So assume, by Lemma~\ref{noIFk-lem}, that $v$ is
precolored $F_j$ for some $j\in \{1,\ldots,\flr{\frac{k+1}2}\}$.  
Lemma~\ref{noFedge-lem} implies that $w\in U_{\ell}$ for some $\ell$.
Form $G'$ from $G-v$ by increasing the number of $F$-neighbors of $w$ by $j$.
Note that $\rho_G(V(G))-\rho_{G'}(V(G'))\le C_{U,\ell}+C_{F,j}-C_E-C_{U,\ell+j}=0$.
(If the new total number of $F$-neighbors of $w$ is at least $k$,
then we color $w$ with $I$.)
  For each $R'\subseteq V(G')$, subgraph $G'[R']$
has an $(I,F_k)$-coloring if and only $G[R'\cup\{v\}]$ does.  Since $G$ is
$(I,F_k)$-critical, so is $G'$.  Since $G'$ is smaller than $G$, by the
minimality of $G$, we have $\rho_{G'}(V(G'))\le -3$.  However, now
$\rho_G(V(G))\le\rho_{G'}(V(G')\le-3$.  Thus, $G$ is not a counterexample.
\end{proof}

Recall, from Section~\ref{potential-sec}, that the heart of any proof using the
potential method is its gap lemmas.  Our next definition plays a crucial role in
the first of these.

\begin{defn}
\label{G'-defn}
Given $R\subsetneq V(G)$ and an $(I,F_k)$-coloring $\vph$ of $G[R]$, we
construct $G':=H(G,R,\vph)$\aside{$G'$, $H(G,R,\vph)$} as follows; see Figure~\ref{G'-fig}.  
Let $\overline{R}:=V(G)\setminus R$.  Let $\nabla(R)\aside{$\overline{R}, \nabla(R)$}:=\{v\in R: \exists w\in
\overline{R}, vw\in E(G)\}$.
To form $G'$ from $G$, delete $R$ and add two new vertices $v_F,v_I$,
where $v_F$ is precolored $F_k$ and $v_I$ is precolored $I$.
(So $G'[\overline{R}]\cong G[\overline{R}]$.)
For each $vw\in E(G)$ with $w\in \overline{R}$, $v\in R$ and
$\vph(v)=F$, add to $G'$ the edge $wv_F$.  
For each $vw\in E(G)$ with $w\in \overline{R}$, $v\in R$ and
$\vph(v)=I$, add to $G'$ the edge $wv_I$.  
Finally, delete $v_F$ or $v_I$ if it has no incident edges.
So $V(G')\subseteq \overline{R}\cup\{v_F,v_I\}$.
In each case, let $X:=V(G')\setminus \overline{R}$.
\end{defn}

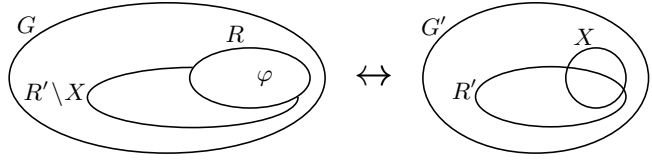
\begin{figure}[!h]
\centering
\begin{tikzpicture}[semithick]
\tikzstyle{uStyle}=[shape = circle, minimum size = 3pt, inner sep = 1pt,
outer sep = 0pt, fill=white, semithick, draw]
\tikzstyle{lStyle}=[shape = circle, minimum size = 5pt, inner sep =
0.5pt, outer sep = 0pt, font=\footnotesize,draw=none,fill=none]
\tikzstyle{vlStyle}=[shape = circle, minimum size = 4pt, inner sep =
1.0pt, outer sep = 0pt, draw, fill=white, semithick, font=\footnotesize]
\tikzset{every node/.style=uStyle}
\begin{scope}[xshift=-3cm]
\draw (-1.97,.7) node[lStyle] {$G$};
\draw (-.1,0) ellipse (2.1cm and 1cm);
\draw (0.8,.58) node[lStyle] {$R$};
\draw (1.0,0) ellipse (.8cm and .4cm); 
\draw (1.2,0) node[lStyle] {$\vph$};
\draw (.24,.14) arc(90:361:1.4cm and .4cm);
\draw (-1.59,-.21) node[lStyle] {$R'\!\setminus\! X$};
\end{scope}
\begin{scope}[xshift=1.8cm]
\draw (-2.15,0) node[lStyle] {\Large{$\rightarrow$}};
\draw (-1.35,.7) node[lStyle] {$G'$};
\draw (0,0) ellipse (1.5cm and 1cm);
\draw (.8,0) ellipse (.4cm and .4cm);
\draw (.65,.55) node[lStyle] {$X$};

\draw (-.945,-.15) node[lStyle] {$R'$};
\draw (.2,-.25) ellipse (1.0cm and .4cm);
\draw (-2.15,0) node[lStyle] {\Large{$\leftrightarrow$}};
\end{scope}

\end{tikzpicture}
\caption{The construction of $G'$ from $G$, $R$, and $\vph$ in
Definition~\ref{G'-defn}, and the vertex subset $R'$ of a critical subgraph of
$G'$ in the proof of the Weak Gap Lemma.  The picture is nearly identical for
the proof of the Strong Gap Lemma.\label{G'-fig}}
\end{figure}
\newpage

\begin{lem}[Weak Gap Lemma]
If $R\subsetneq V(G)$ and $|R|\ge 1$, then $\rho(R)\ge 1$.
\label{weak-gap-lem}
\end{lem}
\begin{proof}
Suppose, to the contrary, that there exists such an $R$ with $\rho(R)\le 0$.
Choose $R$ to minimize $\rho(R)$.  
By Lemma~\ref{noFk-lem}, $F_k=\emptyset$.  So each vertex has positive
potential.  Thus, $|R|\ge 2$ and $R$ induces at least one edge.
Since $G$ is critical, $G[R]$ has an $(I,F_k)$-coloring $\vph$.  Let
$G':=H(G,R,\vph)$.  If $G'$ has an $(I,F_k)$-coloring $\vph'$, then the union
of $\vph$ and $\vph'$ is an $(I,F_k)$-coloring of $G$ (since each edge from $R$
to $\overline{R}$ has endpoints with opposite colors).
So $G'$ has a critical subgraph $G''$; let $R':=V(G'')$ (it is possible that
some vertices in $R'$ have fewer $F$-neighbors in $G''$ than in $G'$).
Note that $|V(G')|\le |V(G)|$ and $|E(G')|<|E(G)|$; thus, $G'$ is smaller than
$G$.  As a result, $G''$ is smaller than $G$.  Thus, $\rho_{G'}(R')\le
\rho_{G''}(R')\le -3$.  Since $G'[X]$ is edgeless, $\rho_{G'}(X')\ge 0$ for
every $X'\subseteq X$.  Now
\begin{align}
\rho_G((R'\setminus X)\cup R) & \le \rho_{G'}(R')-\rho_{G'}(R'\cap
X)+\rho_G(R)
\nonumber \\
& \le -3 + \rho_G(R) \\
& < \rho_G(R). \nonumber
\end{align}
Since $\rho_G((R'\setminus X)\cup R)<\rho_G(R)$ and we chose $R$ to minimize
$\rho_G(R)$, this implies that $(R'\setminus
X)\cup R=V(G)$.  But now $\rho(V(G))\le -3$, so $G$ is not a counterexample.
\end{proof}

The Strong Gap Lemma, which we prove next, is one of the most important lemmas
in the paper.  Very roughly, the proof mirrors that of the Weak Gap Lemma, but
it is much more nuanced, which allows us to prove a far stronger lower bound (one
that grows linearly with $k$).

\begin{lem}[Strong Gap Lemma]
If $R\subsetneq V(G)$ and $G[R]$ contains an edge, 
then $\rho(R) \geq \frac{C_E - 3}2$.
\end{lem}

Before proving the lemma formally, we give a proof sketch.
Choose \Emph{$R$} to minimize $\rho(R)$ among $R\subsetneq V(G)$ such that $G[R]$
contains an edge.
For the sake of contradiction, assume that $\rho(R) < \frac{C_E-3}2$; by
integrality, $\rho(R)\leq \frac{C_E-5}2$.
Let $t := \left\lfloor \frac{\rho(R) +2}3 \right\rfloor$\aside{$t$}.
Again, by integrality, $3t \geq \rho(R)$.
By the Weak Gap Lemma, $t \geq 1$.

We essentially repeat the proof of the Weak Gap Lemma, but more carefully.
In that proof it was crucial that $\rho_{G'}(V(G')\setminus\overline{R})\ge
\rho_G(R)$.  To ensure this now, we will show that $\rho_{G'}(V(G')\setminus
\overline{R})\ge \frac{C_E-5}2$.
To do this, before using induction to get an
$(I,F_k)$-coloring $\vph$ of $G[R]$, we modify $G[R]$ slightly, to get a graph
$G_R$.  Denote $\nabla(R)$ by $v_1, \ldots, v_s$.  
We must ensure that in the coloring $\vph$ of $G[R]$ the components colored $F$ 
containing $v_1,\ldots,v_s$ do not each contain $k$ vertices.  Specifically, if
$F^1,...,F^m$ are the $F$-components of $\vph$ containing vertices 
$v_1,\ldots,v_s$, then we want to maximize $\sum_{j=1}^m(k-|F^j|)$.  When
constructing $G'$, this will allow us to create vertices $v_j$ that are
precolored $F_{|F^j|}$, rather than $F_k$.  When
$j\le \lfloor\frac{k-2}2\rfloor$, recall that $C_{F,k-j}=3j$.  Thus, to ensure that
$\rho(X)\ge \rho(R)$, it suffices to have $\sum_{j=1}^m(k-|F^j|)\ge t$,
since then $\rho(X)\ge \sum_{j=1}^m3(k-|F^j|)\ge 3t \ge \rho(R)$, as desired. 

We construct $G_R$ from $G[R]$ by adding ``fake'' neighbors precolored $F$ to
vertices in $\nabla(R)$; in total, we must add at least $t$ such fake
$F$-neighbors.  More formally, we move vertices from $F_{a_j}$ to $F_{b_j}$
where $\sum b_j = t+\sum a_j$.
The reason that we can color the resulting graph $G_R$ is that we chose $R$ to
minimize $\rho(R)$.  In particular, $\rho_G(Y)\ge \rho_G(R)$ for all $Y\subseteq
R$ (that induces at least one edge).  Thus, $\rho_{G_R}(Y)\ge \rho_G(Y)-3t\ge
\rho_G(R)-3\lfloor\frac{\rho_G(R)+2}3\rfloor\ge-2$.  Thus, $Y$ cannot induce a
critical graph in $G_R$ or some subgraph of it; so, $G_R$ is colorable.
Making all this precise requires more details, which we give below in Case 2.

\begin{proof}
We exactly repeat the first paragraph above; in particular, we define $R$ and
$t$ as above.  Before proceeding to the main case, we handle the easy case that
$\rho(\nabla(R)) < \rho(R)$.

\textbf{Case 1: $\boldsymbol{\rho(\nabla(R)) < \rho(R)}$.}
By our choice of $R$, we know that $G[\nabla(R)]$ is edgeless; also $R\setminus
\nabla(R)\ne\emptyset$.  That is, $\nabla(R)$ is an independent separating set.
Moreover, each vertex of $\nabla(R)$ is colored $F$, since
$\min\{C_{U,k-1},C_I\} \ge \min\{\frac{3k-3}2,\frac{C_E-3}2\}> \frac{C_E-5}2\ge
\rho(R)$. 
Form $\tilde{G}$ from $G$ by moving each vertex of $\nabla(R)$ into $F_k$.
For each $S\subseteq V(G)$ such that $G[S]$ contains an edge, we have
$\rho_{\tilde{G}}(S)\ge \rho_G(S) - \rho_G(\nabla(R)) > \rho_G(S)-\rho_G(R)\ge
0$.  Furthermore, $\rho_{\tilde{G}}(S)\ge 0$ for each $S\subseteq V(G)$ such
that $G[S]$ is edgeless, since each vertex has nonnegative potential.
Thus, every proper induced subgraph of $\tilde{G}$ has an $(I,F_k)$-coloring.
Denote the components of $G - \nabla(R)$ by $C^1, C^2, \ldots, C^r$.
For each $j$, by induction we have an $(I,F_k)$-coloring of $\tilde{G}[C^j
\cup \nabla(R)]$.  The union of these colorings is a coloring of $G$, which
contradicts that $G$ is a counterexample.

\textbf{Case 2: $\boldsymbol{\rho(\nabla(R)) \geq \rho(R)}$}.
Now we show how to form $G_R$ from $G[R]$ so that our $(I,F_k)$-coloring $\vph$
of $G_R$ ensures $\rho_{G'}(V(G')\setminus \overline{R})\ge \rho_G(R)$. 
Denote $\nabla(R)$ by \Emph{$v_1,\ldots,v_s$}.
First suppose that some $v_{\ell}$ is uncolored; say $v_{\ell}\in
U_{p_{\ell}}$.  To form $G_R$ from $G[R]$, we move $v_{\ell}$ to
$U_{p_{\ell}+t}$; if $p_{\ell}+t>k-1$, then we instead move $v_{\ell}$ to $I$. 
(We leave all other vertices in $\nabla(R)$ unchanged.)
Now assume that each $v_j\in \nabla(R)$ is colored $F$.  Say $v_j\in F_{p_j}$
for each $v_j\in \nabla(R)$.
We pick nonnegative integers $\ell_j$ iteratively as follows.  Let $\ell_j :=
\min\{k-p_j, t - \sum_{j''<j} \ell_{j''}\}$.
Note that $\rho(\{v_j\})\le 3(k-p_j)$ for all $j$.  So, if $\sum \ell_j\le t-1$,
then $\rho(\nabla(R))\le 3(t-1)<\rho(R)$; this contradicts the case we are in.
Thus, $\sum \ell_j=t$ (also, $\ell_j \geq 0$ for all $j$).
Form $G_R$ from $G[R]$ by moving each $v_j$ into $F_{p_j+\ell_j}$.

We claim $G_R$ has an $(I,F_k)$-coloring.  Since $G_R$ is smaller than $G$,
this will hold by induction once we show that $\rho_{G_R}(R') \geq -2$ for each $R' \subseteq R$.
Assume, to the contrary, that $\rho_{G_R}(R') \leq -3$, for some $R'$. Now 
$$\rho_G(R') \le \rho_{G_R}(R')+3t \leq -3 + 3t = 3\left(\left\lfloor \frac{\rho_G(R)
+2}3 \right\rfloor - 1\right) < \rho_G(R).$$
By our choice of $R$, this implies that $R'$ is edgeless.  But this contradicts
$\rho_{G_R}(R')\le -3$, since each vertex contributes nonnegative potential.
Thus, $G_R$ has the desired $(I,F_k)$-coloring $\vph$.

We construct $G'$ from $G$, $R$, and $\vph$ as follows.
As described above, $G'$ contains $G[\overline{R}]$, to which we add new
vertices that we call $X$.
Let $F^1, F^2, \ldots, F^m$ denote the components of $F$ in $\vph$ that contain
at least one vertex of $\nabla(R)$.  For each $F^j$, let $(k-\ell_j')$ be the
number of vertices in $F^j$ when $\vph$ is viewed as a coloring of $G[R]$ (not
$G_R$); when constructing $G'$, add to $X$ a vertex $v_{F,j} \in F_{k-\ell_j'}$. 
If $\vph$ uses $I$ on one or more vertices in $\nabla(R)$, then add to $G'$ a
single vertex $v_I \in I$.

Next, we must show that $\rho_{G'}(X) \geq \rho_G(R)$.
Recall that $X$ denotes the vertices in $G'$ that are not in $G$.
By construction, $G'[X]$ is edgeless, so $\rho_{G'}(X) = \sum_{v_j\in
X}\rho_{G'}(v_j)$.
If $v_I\in X$, then $\rho_{G'}(X)\ge \rho_{G'}(\{v_I\}) = C_I =
\frac{C_E-3}2 > \rho_G(R)$, so we are done.  Thus, we assume that $v_I\notin
X$.  Essentially, we want to show that each $v_j\in X\cap F_{k-\ell'_j}$ adds
$3\ell'_j$ to $\rho_{G'}(X)$.  Since $\sum \ell'_j\ge t$, we get $\rho_{G'}(X) =
\sum\rho_{G'}(\{v_j\}) = \sum 3\ell'_j \ge 3t \ge \rho_G(R)$.  But there is a
small complication.  

We only have $\rho_{G'}(\{v_j\}) = 3\ell'_j$ when $\ell'_j\le \lceil \frac{k-3}2
\rceil$; otherwise $\rho_{G'}(\{v_j\})=C_E-3(k-\ell'_j)$, which is $3\ell'_j-1$
when $k$ is even and $3\ell'_j-2$ when $k$ is odd.  If $\ell'_j\ge \lceil
\frac{k-1}2 \rceil$ for at least two values of $j$, then $\rho_{G'}(X)\ge
2(C_E-3(k-\lceil\frac{k-1}2\rceil)) \ge \frac{C_E-5}2\ge \rho_G(R)$, as desired.
So assume that $\ell'_j\ge \lceil\frac{k-1}2\rceil$ for at most one value of $j$.
If $k$ is even, then $\rho_G(R)\le \frac{C_E-5}2 = \frac{3k-6}2$, so
$t = \lfloor \frac{3k-2}6 \rfloor = \lfloor\frac{k-2}2 \rfloor$.  Thus, either 
$\ell'_j\le \lceil\frac{k-1}2\rceil$ for each $j$, or $\sum \ell'_j>t$.  In both
cases, $\rho_{G'}(X)\ge \rho_G(R)$.
Assume instead that $k$ is odd.  If $\rho_G(R)<\frac{C_E-5}2$, then $t\le
\lfloor \frac{k-2}2\rfloor$, and the analysis is similar to that above for $k$
even.  So we instead assume that $\rho_G(R)=\frac{C_E-5}2$ and
$\ell'_{i}=\frac{k-1}2=t$ for some $i$
(with $\ell'_j=0$ for all other $j$).
But in this case, $\rho_{G'}(X)=3t-2$ and $\rho_G(R)=\frac{3k-7}2 =
\frac{3k-3}2-2 = 3t-2$.  So, again $\rho_{G'}(X)\ge \rho_G(R)$, as desired.

The graph $G'$ is smaller than $G$, since by construction $|V(G')| \leq |V(G)|$
(equality may be possible if $G[R] \cong K_{1,s-1}$) and $|E(G')| < |E(G)|$,
since $G[R]$ contains an edge.
Each vertex $v \in \overline{R}$ has at most one neighbor in $R$ 
since otherwise 
$$\rho(R \cup \{v\}) \leq \rho(R) + C_{U,0} - 2C_E \leq \rho(R) -
\frac{C_E + 3}2 \le \frac{C_E-5}2-\frac{C_E+3}2 =-4.$$
If $R\cup \{v\}=V(G)$, then $\rho(V(G))\le -4$, which contradicts that $G$ is a
counterexample.  Otherwise, $R\cup\{v\}\subsetneq V(G)$ and
$\rho(R\cup\{v\})<\rho(R)$, which contradicts our choice of $R$.
So each $v\in\overline{R}$ has at most one neighbor in $R$.
This means that $G'$ does not have an $(I,F_k)$-coloring, since such a coloring
could be combined with $\vph$ to produce an $(I,F_k)$-coloring of $G$.
So $G'$ contains an $(I,F_k)$-critical subgraph $G''$.
Let $W'' := V(G'')$, and by induction $\rho_{G''}(W'') \leq -3$.

Because $G$ is $(I,F_k)$-critical (and thus does not contain proper
$(I,F_k)$-critical subgraphs) $W'' \cap X \neq \emptyset$.
Since $G'[X]$ is edgeless, $\rho_{G'}(X')\ge 0$ for all $X'\subseteq X$.
Let $W := (W''\setminus X) \cup R$.  By submodularity,
\begin{equation}\label{big gap case 2 eqn}
\rho_G(W) \leq \rho_{G'}(W'') - \rho_{G'}(X \cap W'') + \rho_G(R) \leq
(-3) - (0) + \rho_G(R).
\end{equation}
By our choice of $R$, this implies that $W = V(G)$.
We are then in one of two cases, each of which improves the bound in \eqref{big gap case 2 eqn}.
If $X \subset W''$, then $X\cap W''=X$, so we use the prior result that $\rho_{G'}(X) \geq
\rho_G(R)$
to strengthen \eqref{big gap case 2 eqn} and conclude that
$\rho_G(V(G))=\rho_G(W) \le \rho_{G'}(W'') \leq -3$, which is a contradiction.
So assume that $X \setminus W'' \neq \emptyset$.
Because $W = V(G)$, we have $\overline{R} \subset W''$.
By construction, every vertex in $X$ has a neighbor in $\overline{R}$ in $G'$,
and therefore at least one edge with an endpoint in $R$ and the other endpoint
in $\overline{R}$ was not accounted for in \eqref{big gap case 2 eqn}.
Thus, \eqref{big gap case 2 eqn} improves to $\rho_G(W) \leq \rho_G(R) -3 - C_E
\le -\frac{C_E+11}2 < -3$, which is a contradiction.
This finishes Case 2, which completes the proof.
\end{proof}

It will be convenient to write $U^i_j$\aside{$U^i_j, F^i_j$} for the set of
vertices with degree $i$ in $U_j$; similarly for $F^i_j$.  When we do discharging,
vertices in $U^2_j$ will need lots of charge, particularly when $j$ is small.
This motivates our next lemma.  It says that when $j$ is small enough, such
vertices do not exist.

\begin{lem}
\label{getting rid of very small charge}
If $U_j^2\ne \emptyset$, then $j\ge \frac{C_E-7}6$.
\end{lem}
\begin{proof}
Assume, to the contrary, that there exists $j\le \frac{C_E-9}6$ and $v\in U_j^2$.
Denote the neighbors of $v$ by $v_1$ and $v_2$.  Our basic plan is to delete $v$
and add $j+1$ $F$-neighbors to each of $v_1$ and $v_2$; call this new graph $G'$. 
We show that $G'$ has an
$(I,F_k)$-coloring $\vph'$, and extend $\vph'$ to $G$ as follows.  If both
$v_1$ and $v_2$ are colored with $F$, then color $v$ with $I$.  Otherwise,
color $v$ with $F$.  It is easy to see this yields an $(I,F_k)$-coloring of
$G$, a contradiction.  Mainly, we need to show that $\rho_{G'}(R')\ge -2$ for
all $R'\subseteq V(G')$, which we do by the Strong Gap Lemma.  This proves
that $G'$ has the desired $(I,F_k)$-coloring.  We also need to handle the
possibility that our construction of $G'$ creates a component of $F$ with more
than $k$ vertices.

\textbf{Case 1: For each $\bs{v_i\in N(v)}$ either $\bs{v_i\in U}$ or else
$\bs{v_i\in F_{\ell_i}}$ and $\bs{\ell_i+j+1\le k}$.}
We follow the outline above, but need to clarify a few details.  If adding $j+1$
$F$-neighbors to some $v_i\in U$ results in $v_i$ having at least $k$
$F$-neighbors, then we instead color $v_i$ with $I$.  
By design, we do not create any vertices
in $U$ with more than $k-1$ $F$-neighbors or vertices in $F$-components of order
more than $k$.
We also need to check that we do not create any edges with both endpoints colored $I$.
By Lemma~\ref{noI-lem}, no vertex of $G$ is colored $I$.  So we only need to
check that we do not use $I$ on both $v_1$ and $v_2$ when $v_1v_2\in E(G)$.
Suppose that we do.  Assume that $v_1\in U_{\ell_1}$ and $v_2\in U_{\ell_2}$.
So $\ell_1+j+1\ge k$ and $\ell_2+j+1\ge k$.
Now $\rho_G(\{v,v_1,v_2\})=C_{U,\ell_1}+C_{U,\ell_2}+C_{U,j}-3C_E =
\frac{9C_E-9}2-3(j+\ell_1+\ell_2)-3C_E = \frac{3C_E-9}2-3(j+\ell_1+1)-3(\ell_2-1)\le
\frac{C_E-9}2-3(\ell_2-1)\le \frac{C_E-9}2-3(k-2-\frac{C_E-9}6)=C_E-3-3k<-3$. 
This contradicts the Weak Gap Lemma.
Thus, $G'$ has a valid precoloring.

Now we must show that $\rho_{G'}(R')\ge -2$ for all $R'\subseteq V(G')$.
If $G[R']$ is edgeless, then clearly $\rho(R')\ge 0$.  So assume $G[R']$ has at
least one edge.  If $R'\cap N(v)=\emptyset$, then $\rho_{G'}(R')=\rho_G(R')\ge 1$,
by the Weak Gap Lemma.  Instead suppose that $|R'\cap N(v)|=1$.
By the Strong Gap Lemma, $\rho_{G'}(R')\ge \rho_G(R')-3(j+1)\ge \frac{C_E-3}2-3(j+1)\ge
\frac{C_E-3}2-3\frac{C_E-3}6=0$.  Finally, suppose that $|R'\cap N(v)|=2$.  Now
the Weak Gap Lemma (and the fact that $\rho_G(V(G))\ge -2$) gives
\begin{align*}
\rho_{G'}(R')&\ge \rho_G(R'\cup\{v\})+2C_E-C_{U,j}-3(j+1)2\\
&= \rho_G(R'\cup\{v\})+2C_E-(\frac{3C_E-3}2-3j)-6(j+1)\\
&= \rho_G(R'\cup\{v\})+\frac{C_E+3}2-3j-6\\
&\ge \rho_G(R'\cup \{v\})+\frac{C_E}2+\frac{3}2-\frac{C_E-9}2-6\\
&= \rho_G(R'\cup \{v\})\\
&\ge -2.
\end{align*}

\textbf{Case 2: There exists $\bs{v_i\in N(v)}$ such that $\bs{v_i\in F_{\ell_i}}$ and
$\bs{j+\ell_i\ge k}$.}  If $v_1$ and $v_2$ are both precolored $F$, then we simply
delete $v$ (since we can extend $\vph'$ to $G$ by coloring $v$ with $I$).
So, we assume that $v_1\in F_{{\ell}_1}$ with $j+{\ell}_1\ge k$ and $v_2\in
U_{{\ell}_2}$.
Now we simply delete $v$ and color $v_2$ with $F$.  We must again ensure that
$\rho_{G'}(R')\ge -2$ for all $R'\subseteq V(G')$.  If $v_2\notin R'$, then
$\rho_{G'}(R')=\rho_G(R')\ge 1$.  So, assume that $v_2\in R'$.  
If $G'[R']$ is edgeless, then clearly $\rho_{G'}(R')\ge 0$.  So assume that
$G'[R']$ has at least one edge.
Now, similar to above:
\begin{align*}
\rho_{G'}(R')&\ge
\rho_G(R'\cup\{v,v_1\})+2C_E-C_{F,{\ell}_1}-C_{U,j}-C_{U,{\ell}_2}+C_{F,{\ell}_2+1}\\
&\ge
\rho_G(R'\cup\{v,v_1\})+2C_E-3(k-{\ell}_1)-({3C_E-3}-3(j+{\ell}_2))+(C_E-3({\ell}_2+1))\\
&=
\rho_G(R'\cup\{v,v_1\})-3k+3{\ell}_1+3j+3{\ell}_2-3{\ell}_2\\
&=
\rho_G(R'\cup\{v,v_1\})-3k+3(j+{\ell}_1)\\
&\ge\rho(G'\cup\{v,v_1\})\\
&\ge -2.
\end{align*}
\aftermath
\end{proof}

It will turn out that when $j> \frac{C_E-5}6$ vertices in $U_j^2$ will have
nonnegative initial charge.   
By Lemma~\ref{getting rid of very small charge},
we know that $U_j^2=\emptyset$ when $j< \frac{C_E-7}6$.  Thus, to finish the proof
we focus on the vertices in $U^2_j$ when $j=\frac{C_E-5}6$ (in Section~\ref{k-even-sec},
where $k$ is even) and when $j=\frac{C_E-7}6$ (in Section~\ref{k-odd-sec}, where $k$ is odd).

\section{Finishing the Proof when $k$ is Even}
\label{k-even-sec}

Throughout this section, $k$ is always even.
Recall that when $k$ is even $C_E=3k-1$.\aside{$C_E$}
We let $\ell:=\frac{C_E-5}6=\frac{3k-6}6 = \frac{k}2-1$.\aside{$\ell$}

\begin{lem}
\label{no2-threads-k-even}
$G$ does not contain adjacent vertices $v$ and $w$ with $v,w\in U_\ell^2$.
\end{lem}
\begin{proof}
Assume the lemma is false.  Let $v'$ and $w'$ denote the remaining neighbors of
$v$ and $w$, respectively (possibly $v'=w'$).  
By symmetry between $v'$ and $w'$, we assume that $v'\notin F_j$ with $j\ge
k-\ell$ (otherwise $\rho(\{v,w,v',w'\})\le 2C_{F,k-\ell}+2C_{U,\ell}-3C_E =
6\ell+2(\frac{3C_E-3}2-3\ell)-3C_E=-3$, which contradicts the Weak Gap Lemma).
Form $G'$ from $G\setminus\{v,w\}$ by adding $\ell+1$ $F$-neighbors to $v'$.  
If $v'$ now has at least $k$ $F$-neighbors, then move $v'$ to $I$.
(By our assumption on $v'$, we know that $v'$ is not in an $F$-component of
order at least $k+1$.)

Fix $R'\subseteq V(G')$.  If $G'[R']$ has no edges, then $\rho_{G'}(R')\ge 0$,
since each individual vertex has nonnegative potential.  If $v'\notin R'$, then
$\rho_{G'}(R')=\rho_{G}(R')\ge 1$, by the Weak Gap Lemma.  Assume instead that
$v'\in R'$ and $G[R']$ contains at least one edge.  By the Strong Gap Lemma,
$\rho_{G'}(R')\ge \rho_G(R')-3(\ell+1)\ge
\frac{C_E-3}2-3(\ell+1)=\frac{C_E-3}2-\frac{C_E-5+6}2=-2$. 
Thus, by minimality, $G'$ has an $(I,F_k)$-coloring $\vph'$.  

We extend $\vph'$ to $v$ and
$w$ as follows.  If $\vph'(v')=I$, then color $v$ with $F$ and color $w$ with
the color unused on $w'$.  Similarly, if $\vph'(w')=I$, then color $w$ with $F$
and color $v$ with the color unused on $v'$.  (If $\vph'(v')=\vph'(w')=I$, then
$v$ and $w$ lie in an $F$-component with order $2(\ell)+2=\frac{C_E-5}3+2 =
2(\frac{k}2-1)+2=k$.)  Suppose instead that $\vph'(v')=\vph'(w')=F$.  Now color $w$
with $I$ and $v$ with $F$.  Note that this is an $(I,F_k)$-coloring of $G$,
because of the extra $F$-neighbors of $v'$ in $G'$.
\end{proof}

Now we use discharging to show that $G$ cannot exist.
We define our initial charge function so that our assumption $\rho(V(G))\ge -2$
gives an upper bound on the sum of the initial charges. 
(Recall the values of $C_{U,j}$ and $C_{F,j}$ from Definition~\ref{coeff-defn}.
By Lemma~\ref{noI-lem}, $I=\emptyset$.) 
Precisely, let

\begin{itemize}
\item $\ch(v):=C_Ed(v)-2C_{U,j}=C_Ed(v)-2(\frac{3C_E-3}2-3j)$ \aside{$\ch(v)$}\\
\mbox{\!~~~~~~~~~}$= C_E(d(v)-3)+3+6j$ for each $v\in U_j$; and
\item $\ch(v):=C_Ed(v)-2C_{F,j}=C_Ed(v)-2(C_E-3j)$ \\
\mbox{\!~~~~~~~~~}$=C_E(d(v)-2)+6j$ for each $v\in F_j$ with $j\le \ell+1$; and
\item $\ch(v):=C_Ed(v)-2C_{F,j}\ge C_Ed(v)-2(3k-3(\ell+2))$ \\
\mbox{\!~~~~~~~~~}$=C_Ed(v)-3k+6=C_E(d(v)-1)+5$ for each $v\in F_j$ with $j\ge
\ell+2$\\ 
\mbox{\!~~~~~~~~~~~~}(and this inequality is strict when $j>\ell+2$).
\end{itemize}

This definition of $\ch(v)$ yields the inequality
\begin{align}
\sum_{v\in V(G)}\ch(v) = -2\rho(V(G))\le 4.
\label{charge-sum-ineq-even}
\end{align}

\begin{table}[!h]
\centering
$
\begin{array}{c||c|c|c|c|c|c|c|c}
d(v) & U_0 & U_1 & F_1 & U_{\ell} & U_{\ell+1} & U_{\ell+2} &
F_{\ell+2}\\
\hline
1 &   &   &   &   &   &   & 4\\
2 &   &   & 4 &  0 & 2 & 8 \\
3 & 0 & 6 & C_E + 3 \\
4 & C_E-1 & C_E+5 & 2C_E+2 \\
\end{array}
$
\caption{%
Lower bounds on the final charges (when $k$ is even).\label{k-even-charge-table}}
\end{table}
\bigskip

We use a single discharging rule, and let \Emph{$\ch^*(v)$} denote the charge at $v$
after discharging.
\begin{itemize}
\item[(R1)] Each vertex in $U^2_{\ell}$ takes 1 from each neighbor.  
\end{itemize}

\begin{lem}
After discharging by (R1) above, each vertex $v$ with an entry in 
Table~\ref{k-even-charge-table} has $\ch^*(v)$ at least as large charge as shown.  Each
other vertex $v$ has $\ch^*(v)\ge 5$.
\label{k-even-easy-charge-lem}
\end{lem}

\begin{proof}
Note that $\ch^*(v)\ge \ch(v)-d(v)$ for all $v\in V(G)$.  If $v\in U_j$, then
$\ch^*(v)\ge C_E(d(v)-3)+3+6j-d(v)=(C_E-1)(d(v)-3)+6j$.  If $v\in F_j$ and
$j\le \ell+1$, then $\ch^*(v) \ge C_E(d(v)-2)+6j-d(v)=(C_E-1)(d(v)-2)+6j-2$.
If $v\in F_j$ and $j\ge \ell+2$, then $\ch^*(v)\ge C_E(d(v)-1)+5-d(v) =
(C_E-1)(d(v)-1)+4$ (and this inequality is strict when $j>\ell+2$).
By Lemma~\ref{noI-lem}, $I=\emptyset$; by Lemma~\ref{getting rid of very
small charge}, $U^2_j=\emptyset$ when $j<\ell$ .  By Lemma~\ref{delta2-lem},
each $v\in V(G)$ has $d(v)\ge 2$ unless $v\in F^1_j$
with $j\ge \ell+2$.  If $v\in U^2_{\ell+1}$, then $\ch^*(v)\ge
-C_E+1+(C_E-5+6)=2$.  Thus, if $v\notin U_{\ell}^2$, then the lemma follows
from what is above.

By Lemma~\ref{no2-threads-k-even}, if $v\in U^2_{\ell}$, then $v$ does not give
away any charge.  So $v$ finishes with $\ch(v)+2(1) = -C_E+3+6\ell+2(1)
= -C_E+5+(C_E-5)=0$.  
\end{proof}

\begin{cor}
$V(G)\subseteq U^2_{\ell}\cup U^2_{\ell+1}\cup U^3_0\cup U^4_0
\cup F^1_{\ell+2} \cup F^2_1$ (with $U^4_0=\emptyset$ when $k\ge 4$) and
$2|U^2_{\ell+1}|+4|U^4_0|+4|F^1_{\ell+2}| + 4|F^2_1|\le 4$.
\end{cor}
\begin{proof}
This follows directly from Lemma~\ref{k-even-easy-charge-lem}
and~\eqref{charge-sum-ineq-even}.
\end{proof}

\begin{lem}
$G$ has an $(I,F_k)$-coloring, and is thus not a counterexample.
\end{lem}
\begin{proof}
We now construct an $(I,F_k)$-coloring of $G$.
We color each $v\in U_{\ell}^2$ with $I$ and each $v\notin U^2_{\ell}$ with $F$.
By Lemma~\ref{no2-threads-k-even}, we know that $U^2_{\ell}$ is an independent
set. So we only must check that $G-U^2_{\ell}$ is a forest in which each
component has order at most $k$.

Suppose that $G-U^2_{\ell}$ contains a cycle, $C$.  Clearly $C$ has no vertex in
$U^4_0\cup F_1^2$, since such a vertex would end with charge at least 6, a contradiction.
(Also, $C$ has no vertex in $F^1_{\ell+2}$.)
Furthermore, each vertex in $U^2_{\ell+1}\cup U^3_0$ on such a cycle would end
with charge at least 2.  Since $G$ is simple, $C$ has length at least 3, so its
vertices end with charge at least 6, a contradiction.  Thus, $G-U^2_{\ell}$ is
acyclic.  If $U^4_0\cup F^1_{\ell+2} \cup F_1^2\ne \emptyset$, then $U^2_{\ell+1}=\emptyset$
and $|U^4_0\cup F^1_{\ell+2} \cup F^2_1|=1$.  Furthermore, $G$ is a bipartite graph with
$U^2_{\ell}$ as one part and $U_0^3\cup U^4_0 \cup F^1_{\ell+2} \cup F^2_1$ as
another (otherwise $G$ has total charge at least 5, a contradiction).  So
$G$ has an $(I,F_k)$-coloring using $I$ on $U^2_{\ell}$ and $F$ on
$U^3_0\cup U^4_0 \cup F^1_{\ell+2} \cup F^2_1$.

Assume instead that $U^4_0 \cup F^1_{\ell+2} \cup F^2_1=\emptyset$.  Recall
that $G-U^2_{\ell}$ is a forest.
Let $T$ denote a component of this forest, let $n_2:=|U^2_{\ell+1}\cap V(T)|$,
and let $n_3:=|U^3_0\cap V(T)|$.  The number of edges incident to $T$ is
$(\sum_{v\in V(T)}d(v))-2|E(T)|=2n_2+3n_3-2(n_2+n_3-1)=n_3+2$.  Recall that $T$ gives
away 1 along each such edge.  Each vertex counted by $n_3$ begins with 3, and
each vertex counted by $n_2$ begins with 4.  Thus the total final charge of
vertices of $T$ is $4n_2+3n_3-(n_3+2) = 4n_2+2n_3-2$.  Since $G$ has total
charge at most 4, either $n_2=1$ and $n_3\le 1$ or else $n_2=0$ and $n_3\le 3$.  
Now color all vertices of $T$ with $F$, except when $n_2=0$, $n_3=3$, and
$k=2$.  In that case, the total final charge of $T$ is 4, so every other
component of $G-U^2_{\ell}$ is an isolated vertex in $U^3_0$.  Now color the
leaves of $T$ with $F$ and the center vertex, say $v$, with $I$. Also
recolor the neighbor of $v$ outside of $T$ with $F$.
\end{proof}

\section{Finishing the Proof when $k$ is Odd}
\label{k-odd-sec}
\subsection{Reducible Configurations when $k$ is Odd}
\label{k-odd-reduc-sec}
Throughout this section, $k$ is always odd.
Recall that when $k$ is odd $C_E=3k-2$.\aside{$C_E$}
Further, let $\ell:=\frac{C_E-7}6=\frac{3k-9}6=\frac{k-3}2$.\aside{$\ell$}
(Note that $C_E$ and $\ell$ are defined differently from the previous section.)
We will frequently use the fact that $2\ell+3=k$.

\begin{lem}
\label{no2-threads-k-odd}
$G$ does not contain adjacent vertices $v$ and $w$ with $v\in U_\ell^2$ and $w\in
U_\ell^2\cup U_{\ell+1}^2$.
\end{lem}
\begin{proof}
Assume the lemma is false.  Let $v'$ and $w'$ denote the remaining neighbors of
$v$ and $w$, respectively (possibly $v'=w'$).  Form $G'$ from
$G\setminus\{v,w\}$ by adding $\ell+1$ $F$-neighbors to $v'$.  
(Suppose this puts $v'$ in an $F$-component of order at least $k+1$.  
In this case, $\rho(\{v',v,w\})\le C_{F,k-\ell}+C_{U,\ell}+C_{U,\ell+1}-2C_E
= 3\ell+(\frac{3C_E-3}2-3\ell)+(\frac{3C_E-3}2-3(\ell+1))-2C_E = 3C_E-3-2C_E-3(\ell+1)
= C_E-3-3(\frac{C_E-7}6+1) = \frac{C_E-5}2$, which contradicts the Strong Gap Lemma. 
So $v'$ is not in an $F$-component of order at least $k+1$.)

Now we show that $\rho_{G'}(R')\ge -2$ for all $R'\subseteq V(G')$.  Fix some
$R'\subseteq V(G')$.  If $v'\notin R'$, then $\rho_{G'}(R')=\rho_{G}(R')\ge
1$, by the Weak Gap Lemma.  If $G'[R']$ has no edges, then $\rho_{G'}(R')\ge
0$, since each coefficient in Definition~\ref{rho-defn} is nonnegative.  
Assume instead that $v'\in R'$ and $G[R']$ has at least one edge.
By the Strong Gap Lemma, $\rho_{G'}(R')\ge \rho_G(R')-3(\ell+1)\ge
\frac{C_E-3}2-3(\ell+1)=\frac{C_E-3}2-\frac{C_E-7+6}2=-1$. 
Thus, $G'$ has an $(I,F_k)$-coloring $\vph'$.  

We extend $\vph'$ to $v$ and
$w$ as follows.  If $\vph'(v')=I$, then color $v$ with $F$ and color $w$ with
the color unused on $w'$.  Similarly, if $\vph'(w')=I$, then color $w$ with $F$
and color $v$ with the color unused on $v'$.  (If $\vph'(v')=\vph'(w')=I$, then
$v$ and $w$ lie in an $F$-component with order at most $2(\ell)+3=\frac{C_E-7}3+3 =
\frac{3k-9}3+3=k$.)  Suppose instead that $\vph'(v')=\vph'(w')=F$.  Now color $w$
with $I$ and $v$ with $F$.  Note that this is an $(I,F_k)$-coloring of $G$,
because of the extra $F$-neighbors of $v'$ in $G'$.
\end{proof}

\begin{lem}
\label{no3withall2s-lem}
$G$ does not contain a vertex $v\in U_0^3$ with all three neighbors in
$U_\ell^2$.
\end{lem}
\begin{proof}
Suppose the lemma is false.  Form $G'$ from $G$ by deleting $v$ and its three
2-neighbors.  Since $G$ is critical, $G'$ has an $(I,F_k)$-coloring $\vph'$.
Now we extend $\vph'$ to all of $G$.  Color each 2-neighbor of $v$ with the
color unused on its neighbor in $G'$.  If all three 2-neighbors of $v$ are
colored $F$, then color $v$ with $I$.  Otherwise, color $v$ with $F$.  
This produces an $(I,F_k)$-coloring of $G$ (because $2\ell+3=k$).
\end{proof}

\begin{lem}
\label{noadj3swith32nbrs-lem}
$G$ does not contain adjacent vertices $v, w\in U_0^3$ such that $v$ has two
neighbors in $U^2_\ell$ and $w$ has at least one neighbor in $U^2_\ell$.
\end{lem}
\begin{proof}
Suppose the lemma is false.  Denote the 2-neighbors of $v$ by $x$ and $y$,
and denote a 2-neighbor of $w$ in $U^2_{\ell}$ by $z$.  Denote by $w'$, $x'$,
$y'$, and $z'$ the remaining neighbors of $w$, $x$, $y$, and $z$ (other than
$v$, $w$, and $z$); see Figure~\ref{noadj3swith32nbrs-fig}.  
We want to form $G'$ from $G$ by deleting $y$ and contracting both edges
incident to $z$; however, this creates parallel edges when $w'z'\in E(G)$, so
we consider two cases.  Before doing that, we briefly consider the possibility
that $y=z$.

If $y=z$, then by criticality we color $G-\{v,w,x,y/z\}$. To extend the coloring
to $G$, we color $w$ with the color unused on $w'$ and color $x$ with the color
unused on $x'$.  If both $w$ and $x$ are colored $F$, then we color $v$ with
$I$; otherwise, we color $v$ with $F$.  Finally, if both $v$ and $w$ are colored
$F$, then we color $y/z$ with $I$; otherwise, we color $y/z$ with $F$.  It is
easy to check that this coloring has no cycle colored $F$ and no edge with both
endpoints colored $I$.  It also has no
$F$-component of size larger than $2\ell+3=k$.  Thus, we assume $y\ne z$.

\textbf{Case 1: $\bs{w'z'\notin E(G)}$.}
Form $G'$ from $G$ by deleting $y$ and contracting both edges
incident to $z$; the new vertex $w\!*\!z'$ formed from 
$w$ and $z'$ inherits the precoloring of $z'$.  

Consider $R'\subseteq V(G')$.  
If $w*z'\notin R'$, then $\rho_{G'}(R')=\rho_G(R')\ge 1$, by the Weak Gap Lemma.
If $G'[R']$ has no edges, then $\rho_{G'}(R')\ge 0$, since each individual
vertex has nonnegative potential.  
So assume that $w*z'\in R'$ and $G'[R']$ has at least one edge.  Now

\begin{align*}
\rho_{G'}(R')&=\rho_G((R'\setminus\{w*z'\})\cup\{w,z,z'\})-C_{U,\ell}-C_{U,0}+2C_E\\
&=\rho_G((R'\setminus\{w*z'\})\cup\{w,z,z'\})-(3C_E-3-3\ell)+2C_E\\
&=\rho_G((R'\setminus\{w*z'\})\cup\{w,z,z'\})-C_E+3+\frac{C_E-7}2\\
&=\rho_G((R'\setminus\{w*z'\})\cup\{w,z,z'\})-\frac{C_E+1}2\\
&\ge -2, 
\end{align*}
where the final inequality holds because the Strong Gap Lemma gives
$\rho_G(R'\setminus\{w*z'\})\cup\{w,z,z'\})\ge \frac{C_E-3}2$. 
Thus, $G'$ has an $(I,F_k)$-coloring $\vph'$.  

\textbf{Case 2: $\bs{w'z'\in E(G)}$.}
Again form $G'$ from $G$ by deleting $y$ and contracting both edges incident to
$z$; the new vertex $w*z'$ formed from $w$ and $z'$ inherits the precoloring of
$z'$.  Since $w'z'\in E(G)$, this creates parallel edges between $w'$ and
$w*z'$.  If one of $w'$ and $w*z'$ is colored with $F$, then delete both of the
parallel edges and color the other endpoint with $I$.
(By Lemma~\ref{noFedge-lem}, at least one of $w'$ and $z'$ is not colored $F$. )
If neither $w'$ nor $w*z'$ is colored with $F$, then we delete one edge between
$w'$ and $w*z'$ and add $\frac{k-1}2$ $F$-neighbors to each of them.
(It is not possible that each of $w'$ and $w*z'$ ends with at least $k$
$F$-neighbors, so gets recolored $I$, since in that case $\rho_G(\{w,z,w',z'\})$
violates the Strong Gap Lemma.)

Now we must show that $\rho_{G'}(R')\ge -2$ for all $R'\subseteq V(G')$.  
If $R'\cap\{w',w*z'\}=\emptyset$, then $\rho_{G'}(R')=\rho_G(R')\ge 1$ by the
Weak Gap Lemma.  So, we assume that $R'\cap\{w',w*z'\}\ne\emptyset$.
We will compute $\rho_{G'}(R')-\rho_G((R'\setminus\{w*z'\})\cup\{w,z,w',z'\})$.
For convenience, let $\alpha:=-2C_{U,0}-C_{U,\ell}+C_I+3C_E$.
We have 5 cases to consider.  
\begin{enumerate}
\item
We added $F$-neighbors to both $w'$ and $w*z'$ and $|R'\cap\{w',w*z'\}|=2$.
Now $\rho_{G'}(R')-\rho_G((R'\setminus\{w*z'\})\cup\{w,z,w',z'\}) = 
-C_{U,0}-C_{U,\ell}+3C_E-3(k-1) = \alpha+C_{U,0}-C_I-3(k-1) =
\alpha+\frac{3C_E-3}2-\frac{C_E-3}2-(C_E-1)=\alpha+1$.
\item
We added $F$-neighbors to both $w'$ and $w*z'$ and $|R'\cap\{w',w*z'\}|=1$.
Now $\rho_{G'}(R')-\rho_G((R'\setminus\{w*z'\})\cup\{w,z,w',z'\}) \ge 
-2C_{U,0}-C_{U,\ell}+4C_E-\frac{3k-3}2 = \alpha+C_E-C_I-\frac{3k-3}2 =
\alpha+C_E-\frac{C_E-3}2-\frac{C_E-1}2=\alpha+2$.
\item
We moved $w'$ or $w*z'$ to $I$ and $|R'\cap\{w',w*z'\}|=2$.
Now $\rho_{G'}(R')-\rho_G((R'\setminus\{w*z'\})\cup\{w,z,w',z'\}) \ge 
-2C_{U,0}-C_{U,\ell}+C_I+3C_E = \alpha$.
\item
We moved $w'$ or $w*z'$ to $I$ and $R'$ contains the one we moved to $I$,
but not the other.
Now $\rho_{G'}(R')-\rho_G((R'\setminus\{w*z'\})\cup\{w,z,w',z'\}) \ge 
-2C_{U,0}-C_{U,\ell}-C_{F,1}+C_I+4C_E = \alpha-C_{F,1}+C_E \ge \alpha+3$.
\item
We moved $w'$ or $w*z'$ to $I$ and $R'$ contains the one we did not move to
$I$, but not the other.
Now $\rho_{G'}(R')-\rho_G((R'\setminus\{w*z'\})\cup\{w,z,w',z'\}) \ge 
-2C_{U,0}-C_{U,\ell}+4C_E = \alpha-C_I+C_E > \alpha$.
\end{enumerate}
Note that $\alpha = -2C_{U,0}-C_{U,\ell}+3C_E+C_I =
-3C_E+3-(C_E+2)+3C_E+\frac{C_E-3}2 = -\frac{C_E+1}2$.
Now, by the Strong Gap Lemma, $\rho_{G'}(R')\ge
\rho_G((R'\setminus\{w*z'\})\cup\{w,z,w',z'\})-\frac{C_E+1}2 =
\frac{C_E-3}2-\frac{C_E+1}2=-2$.
Thus, $G'$ again has an $(I,F_k)$-coloring $\vph'$.  
\smallskip

We will show how to
extend $\vph'$ to $G$ (after possibly modifying it a bit).  
We first extend $\vph'$ to an $(I,F_k)$-coloring $\vph$ of
$G-y$ by uncontracting the two edges incident to $z$, coloring both $w$ and
$z'$ with $\vph'(w\!*\!z')$, and coloring $z$ with the opposite color.

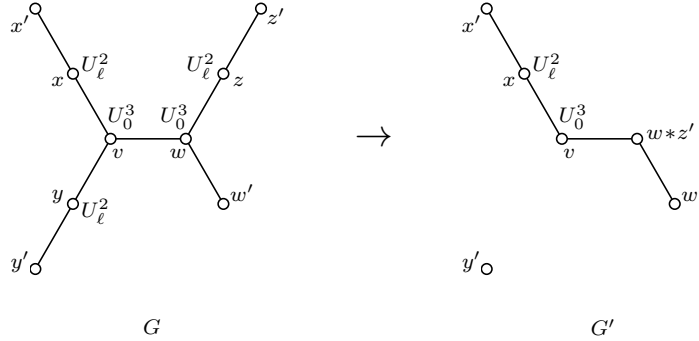
\begin{figure}[!h]
\centering
\begin{tikzpicture}[semithick]
\tikzset{every node/.style=uStyle}
\tikzstyle{l2Style}=[shape = rectangle, draw=none, inner sep=0pt, outer sep=0pt]

\draw (0,0) node (v) {} -- (1,0) node (w) {} (v) -- ++(120:1cm) node (x) {} --
++ (120:1cm) node (x') {} (v) -- ++ (240:1cm) node (y) {} -- ++(240:1cm) node
(y') {} (w) --++(60:1cm) node (z) {} --++(60:1cm) node (z') {} (w) --++(300:1cm)
node (w') {};
\draw (v) ++(.1cm,-.2cm) node[l2Style] {\scriptsize{$v$}};
\draw (v) ++(.15cm,.3cm) node[l2Style] {\scriptsize{$U_0^3$}};
\draw (w) ++(-.1cm,-.2cm) node[l2Style] {\scriptsize{$w$}};
\draw (w) ++(-.15cm,.3cm) node[l2Style] {\scriptsize{$U_0^3$}};
\draw (x) ++(-.2cm,-.1cm) node[l2Style] {\scriptsize{$x$}};
\draw (x) ++(.3cm,.1cm) node[l2Style] {\scriptsize{$U^2_{\ell}$}};
\draw (x') ++(-.2cm,-.2cm) node[l2Style] {\scriptsize{$x'$}};
\draw (y) ++(-.2cm,.1cm) node[l2Style] {\scriptsize{$y$}};
\draw (y) ++(.3cm,-.1cm) node[l2Style] {\scriptsize{$U^2_{\ell}$}};
\draw (y') ++(-.2cm,.1cm) node[l2Style] {\scriptsize{$y'$}};
\draw (z) ++(.2cm,-.1cm) node[l2Style] {\scriptsize{$z$}};
\draw (z) ++(-.3cm,.1cm) node[l2Style] {\scriptsize{$U^2_{\ell}$}};
\draw (z') ++(.2cm,-.1cm) node[l2Style] {\scriptsize{$z'$}};
\draw (w') ++(.25cm,.15cm) node[l2Style] {\scriptsize{$w'$}};
\draw (.55cm,-2.5) node[l2Style] {\scriptsize{$G$}};

\draw (3.5,0) node[lStyle] {\Large{$\rightarrow$}};

\begin{scope}[xshift=6cm]
\draw (0,0) node (v) {} -- (1,0) node (w) {} (v) -- ++(120:1cm) node (x) {} --
++ (120:1cm) node (x') {} (v) ++(240:2cm) node
(y') {} (w) --++(300:1cm) node (w') {};
\draw (v) ++(.1cm,-.2cm) node[l2Style] {\scriptsize{$v$}};
\draw (v) ++(.15cm,.3cm) node[l2Style] {\scriptsize{$U_0^3$}};
\draw (w) ++(.45cm,.1cm) node[l2Style] {\scriptsize{$w\!*\!z'$}};
\draw (x) ++(-.2cm,-.1cm) node[l2Style] {\scriptsize{$x$}};
\draw (x) ++(.3cm,.1cm) node[l2Style] {\scriptsize{$U^2_{\ell}$}};
\draw (x') ++(-.2cm,-.2cm) node[l2Style] {\scriptsize{$x'$}};
\draw (y') ++(-.2cm,.1cm) node[l2Style] {\scriptsize{$y'$}};
\draw (z) ++(.2cm,-.1cm) node[l2Style] {\scriptsize{$z$}};
\draw (z) ++(-.3cm,.1cm) node[l2Style] {\scriptsize{$U^2_{\ell}$}};
\draw (z') ++(.2cm,-.1cm) node[l2Style] {\scriptsize{$z'$}};
\draw (w') ++(.25cm,.15cm) node[l2Style] {\scriptsize{$w'$}};
\draw (.55cm,-2.5) node[l2Style] {\scriptsize{$G'$}};
\end{scope}
\end{tikzpicture}
\caption{Forming $G'$ from $G$ in the proof of
Lemma~\ref{noadj3swith32nbrs-lem}.\label{noadj3swith32nbrs-fig}}
\end{figure}

Suppose that $\vph(y')=I$.  
If $\vph(v)=I$, then we color $y$ with $F$ and are done.  So assume $\vph(v)=F$.
If $\vph(w)=\vph(x)=I$, then we again color $y$ with $F$ and are done.  
If $\vph(w)=\vph(x)=F$, then we recolor $v$ with $I$ and are done as above.  
So assume
that exactly one of $w$ and $x$ uses $I$ in $\vph$ and the other uses $F$. 
First suppose that $\vph(w)=I$ and $\vph(x)=F$.  If $\vph(x')=I$, then we color
$y$ with $F$ and are done.  Instead assume that $\vph(x')=F$.  Now we recolor
$x$ with $I$ and color $y$ with $F$.  Thus, we assume instead that
$\vph(w)=F$ and $\vph(x)=I$.  If both neighbors of $w$ other than $v$ are
colored $I$, then we color $y$ with
$F$ and are done.  So assume that $z$ is the only neighbor of $w$ colored $I$.
Let $s_1$ and $s_2$ denote the orders of the $F$-components of $\vph$ that
contain $w$ and $z'$, respectively.  If $s_1\le k-(\ell+1)$, then we color $y$
with $F$.  If $s_2\le k-(\ell+1)$, then we recolor $z$ with $F$, recolor $w$
with $I$, and color $y$ with $F$.  
The key observation is that one of these two inequalities
must hold.  Suppose not.  The $F$-component in $\vph'$ containing $w*z$
shows that $k\ge s_1+s_2-1$.  If both inequalities above fail, then $k\ge
s_1+s_2-1 \ge (k-(\ell+1)+1)+(k-(\ell+1)+1)-1 = 2k-2\ell-1 = 2k-(k-3)-1=k+2$,
which is a contradiction.

Suppose instead that $\vph(y')=F$.
If $\vph(v)=F$, then we color $y$ with $I$ and are done.  Assume instead that
$\vph(v)=I$.  First suppose $w'$ and $z$ are colored $I$.  Now
recolor $v$ with $F$ and color $y$ with $I$; finally, if $x'$ is colored $F$,
then recolor $x$ with $I$.  
This gives an $(I,F_k)$-coloring of $G$. 
Suppose instead that $w'$ is colored $F$.  
Let $s_1$ and $s_2$ denote the orders of the $F$-components of $\vph$ that
contain $w$ and $z'$, respectively.  
Suppose that $s_1\le k-(\ell+2)$. 
Color $y$ with $I$, recolor $v$ with $F$, and if $\vph(x')=F$, then recolor $x$ with $I$.  
This gives an $(I,F_k)$-coloring of $G$. 
Suppose instead that $s_2\le k-(\ell+1)$. 
Again color $y$ with $I$, recolor $v$ with $F$, and if $\vph(x')=F$, then recolor $x$ with $I$.  
Finally, recolor $w$ with $I$ and recolor $z$ with $F$.
Again, this gives an $(I,F_k)$-coloring of $G$. 
The key observation is that one of these two inequalities
must hold; the proof is identical to that in the previous paragraph, except that
the first inequality is tighter by 1.
\end{proof}

\subsection{Discharging when $k$ is Odd}
\label{k-odd-discharging-sec}
Now we use discharging to show that $G$ cannot exist.
It is helpful to remember that $I=\emptyset$, by Lemma~\ref{noI-lem}, and
$U_j^2=\emptyset$ when $j<\ell$, by Lemma~\ref{getting rid of very small charge}. 
Furthermore, by Lemma~\ref{delta2-lem}, each $v\in V(G)$ satisfies $d(v)\ge 2$
unless $v\in F_j$ with $j\ge \frac{k+3}2$.  
We define our initial charge function so that our assumption $\rho(V(G))\ge -2$
gives an upper bound on the sum of the initial charges. 
(Recall the values of $C_{U,j}$ and
$C_{F,j}$ from Definition~\ref{coeff-defn}.) Precisely, let

\begin{itemize}
\item $\ch(v):=C_Ed(v)-2C_{U,j}=C_Ed(v)-2(\frac{3C_E-3}2-3j)$ \aside{$\ch(v)$}\\
\mbox{\!~~~~~~~~~}$= C_E(d(v)-3)+3+6j$ for each $v\in U_j$; and
\item $\ch(v):=C_Ed(v)-2C_{F,j}=C_Ed(v)-2(C_E-3j)$ \\
\mbox{\!~~~~~~~~~}$=C_E(d(v)-2)+6j$ for each $v\in F_j$ with $j\le \frac{k+1}2$; and
\item $\ch(v):=C_Ed(v)-2C_{F,j}\ge C_Ed(v)-2(3k-\frac{3k+9}2)$ \\
\mbox{\!~~~~~~~~~}$=C_Ed(v)-3k+9=C_E(d(v)-1)+7$ for each $v\in F_j$ with $j\ge \frac{k+3}2$.
\end{itemize}

This definition of $\ch(v)$ yields the inequality
\begin{align}
\sum_{v\in V(G)}\ch(v) = -2\rho(V(G))\le 4.
\label{charge-sum-ineq-odd}
\end{align}

\begin{table}[!h]
\centering
$
\begin{array}{c||c|c|c|c|c|c|c|c}
d(v) & U_0 & U_1 & U_2 & F_1 & F_2 & U_{\ell} & U_{\ell+1} & U_{\ell+2}\\
\hline
2 &   &   &   & 2 & 8 & 0 & 0 & 4\\
3 & 0 & 3 & 9 & C_E & C_E+6\\
4 & C_E-5 & C_E+1 & C_E+7 & 2C_E-2 
\end{array}
$
\caption{%
Lower bounds on the final charges (when $k$ is odd).\label{k-odd-charge-table}}
\end{table}
\bigskip

We use two discharging rules, and let \Emph{$\ch^*(v)$} denote the charge at $v$
after discharging.
\begin{enumerate}
\item[(R1)] Each $v\in U_\ell^2$ (2-vertex) takes 2 from each neighbor.
\item[(R2)] Each $v\in U_0^3$ (3-vertex) with two neighbors in $U_\ell^2$ takes 1 from
its other neighbor.
\end{enumerate}

\begin{lem}
After discharging with rules (R1) and (R2) above, each vertex $v$ with an entry in 
Table~\ref{k-odd-charge-table} has $\ch^*(v)$ at least as large charge as shown.  Each
other vertex $v$ has $\ch^*(v)\ge 5$.
\label{easy-charge-lem}
\end{lem}

\begin{proof}
Note that $\ch^*(v)\ge \ch(v)-2d(v)$ for all $v\in V(G)$.
If $v\in U_j$, then $\ch^*(v)\ge C_E(d(v)-3)+3+6j-2d(v)=(C_E-2)(d(v)-3)+6j-3$.
If $v\in F_j$ and $j\le \frac{k+1}2$, then $\ch^*(v) \ge
C_E(d(v)-2)+6j-2d(v)=(C_E-2)(d(v)-2)+6j-4$.
If $v\in F_j$ and $j\ge \frac{k+3}2$, then $\ch^*(v)\ge
C_E(d(v)-1)+7-2d(v)=(C_E-2)(d(v)-1)+5$.
If $v\notin U_{\ell}^2\cup U_{\ell+1}^2\cup U_0^3$, 
then the lemma follows from what is above.

If $v\in U_{\ell}^2$, then $v$ has no neighbors in $U_{\ell}^2\cup
U^2_{\ell+1}$, by Lemma~\ref{no2-threads-k-odd}.  Thus, $\ch^*(v)=-4+2(2)=0$.
If $v\in U_{\ell+1}^2$, then $v$ has no neighbors in $U^2_{\ell}$, by
Lemma~\ref{no2-threads-k-odd}.  Thus, $\ch^*(v)\ge 2-2(1)=0$.  Finally, suppose that
$v\in U_0^3$.  By Lemma~\ref{no3withall2s-lem}, $v$ does not have three
neighbors in $U^2_{\ell}$.  A vertex in $U_0^3$ is \Emph{needy} if it has two
neighbors in $U^2_{\ell}$.  By Lemma~\ref{noadj3swith32nbrs-lem}, a vertex in
$U_0^3$ cannot have both a neighbor in $U^2_{\ell}$ and a needy 3-neighbor. 
Thus, we have $\ch^*(v)\ge \min\{3-2,3-2(2)+1,3-3(1)\}=0$.
\end{proof}

\begin{cor}
$V(G)=U^2_{\ell}\cup U^2_{\ell+1}\cup U^2_{\ell+2}\cup F_1^2\cup U_0^3\cup U_1^3\cup U_0^4$.
Furthermore $4|U^2_{\ell+2}|+2|F_1^2|+3|U_1^3|+(C_E-5)|U_0^4|\le 4$.  (In
particular, $U_0^4=\emptyset$ when $k\ge 5$.)
\label{few-vertex-types-cor}
\end{cor}
\begin{proof}
This corollary follows directly from Lemma~\ref{easy-charge-lem}
and~\eqref{charge-sum-ineq-odd}.
\end{proof}

If we knew that $\sum_{v\in V(G)}\ch(v)<0$, then Lemma~\ref{easy-charge-lem}
would yield a contradiction.  However, we only know that $\sum_{v\in
V(G)}\ch(v)\le 4$, so we are not done yet.  We will now try to construct the
desired coloring.  We show that we can do this unless $\sum_{v\in V(G)}\ch(v)>
4$, which gives the desired contradiction.  Our basic plan is to color all of 
$U^2_{\ell}$ with $I$.  This will force all neighbors of $U^2_{\ell}$ into $F$.
Furthermore, all but a constant number of vertices in $V(G)\setminus
U^2_{\ell}$ will go into $F$.  To do this, we consider the components of
$G\setminus U^2_{\ell}$.  All but a constant number of these have size at most
4, and all have size at most 8.

\begin{lem}
Each component of $G\setminus U^2_{\ell}$ is one of the 30 shown below in
Figures~\ref{cases1234-fig}-\ref{case6T678-fig}, and has final charge as shown.
 (The coloring of vertices as black and white can be ignored for now.)
\end{lem}
\begin{proof}
Let $J$ be a component of $G\setminus U_{\ell}^2$.  Let $\ch^*(J):=\sum_{v\in
V(J)}\ch^*(v)$.  We will prove that if $J$ is some component other than one of
those shown, then either $G$ contains a reducible configuration or $\ch^*(J)>4$;
both possibilities yield a contradiction.

\textbf{Case 1: $\boldsymbol{V(J)\cap U_0^4\ne \emptyset}$.}
(By Corollary~\ref{few-vertex-types-cor}, this is possible only when $k=3$.)
Assume $v\in V(J)\cap U_0^4$.  If $V(J)=\{v\}$, then we are done.  Otherwise,
$\ch^*(v)\ge 3$.  So, by Table~\ref{k-odd-charge-table}, we know $V(J)\setminus
\{v\}\subseteq U_0^3\cup U_{\ell+1}^2$.  Let $w$ be a neighbor of $v$ in $J$.  If
$w\in U_{\ell+1}^2$, then $\ch^*(J)\ge \ch^*(v)+\ch^*(w)\ge 4+1$, a contradiction. 
The same is true if $w\in U_0^3$ unless $w$ is needy (recall that $w$ cannot
have both a neighbor in $U_{\ell}^2$ and a needy 3-neighbor, by
Lemma~\ref{noadj3swith32nbrs-lem}).  If $v$ has at most two needy 3-neighbors,
then we are done.  Otherwise, $\ch^*(v)\ge 5$, a contradiction.

\textbf{Case 2: $\boldsymbol{V(J)\cap U_{\ell+2}^2\ne \emptyset}$.}
Assume $v\in V(J)\cap U_{\ell+2}^2$.
If $V(J)=\{v\}$, then we are done.  Otherwise, $\ch^*(v)\ge 5$, a contradiction.

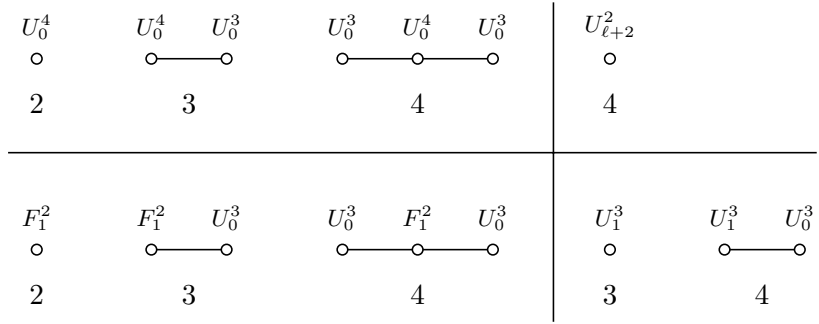
\begin{figure}[!htb]
\centering
\begin{tikzpicture}[semithick]
\tikzstyle{IStyle}=[shape = circle, minimum size = 6.5pt, draw=none, fill=white,
outer sep=0pt, inner sep=2.5pt]
\begin{scope}[xshift=.15in]
\draw node[u04Style] {};
\draw (0,-.6) node[lStyle] {2};
\end{scope}

\begin{scope}[xshift=.75in]
\draw (0,0) node[u04Style] {} -- (1,0) node[u03Style] {};
\draw (.5,-.6) node[lStyle] {3};
\end{scope}

\begin{scope}[xshift=1.75in]
\draw (0,0) node[u03Style] {} -- (1,0) node[u04Style] {} -- (2,0) node[u03Style] {};
\draw (1,-.6) node[lStyle] {4};
\end{scope}

\begin{scope}[xshift=3.15in]
\draw node[u22Style] {};
\draw (0,-.6) node[lStyle] {4};
\end{scope}

\draw[] (0,-1.25) -- (10.75,-1.25) (7.25,.75) -- (7.25,-3.5);

\begin{scope}[xshift=.15in, yshift=-1in]
\draw node[f12Style] {};
\draw (0,-.6) node[lStyle] {2};
\end{scope}

\begin{scope}[xshift=.75in, yshift=-1in]
\draw (0,0) node[f12Style] {} -- (1,0) node[u03Style] {};
\draw (.5,-.6) node[lStyle] {3};
\end{scope}

\begin{scope}[xshift=1.75in, yshift=-1in]
\draw (0,0) node[u03Style] {} -- (1,0) node[f12Style] {} -- (2,0) node[u03Style] {};
\draw (1,-.6) node[lStyle] {4};
\end{scope}

\begin{scope}[xshift=3.15in, yshift=-1in]
\draw node[u13Style] {};
\draw (0,-.6) node[lStyle] {3};
\end{scope}

\begin{scope}[xshift=3.75in, yshift=-1in]
\draw (0,0) node[u13Style] {} -- (1,0) node[u03Style] {};
\draw (.5,-.6) node[lStyle] {4};
\end{scope}

\end{tikzpicture}
\caption{The 9 possible components of $G\setminus U_0^2$ in Cases 1--4.\label{cases1234-fig}}
\end{figure}

\textbf{Case 3: $\boldsymbol{V(J)\cap F_1^2\ne \emptyset}$.}
Assume $v\in V(J)\cap F_1^2$.  If $V(J)=\{v\}$, then we are done.  Otherwise,
let $w$ be a neighbor of $v$ in $J$. If $w\in U_{\ell+1}^2$, then $\ch^*(J)\ge
\ch^*(v)+\ch^*(w)\ge 4+1$, a contradiction.  Thus, we must have $w\in U_0^3$.
If $w$ is not needy, then $\ch^*(v)+\ch^*(w)\ge 4+1$, a contradiction.
Thus, $v$ has one or two needy neighbors (and this is all of $J$).

\textbf{Case 4: $\boldsymbol{V(J)\cap U_1^3\ne \emptyset}$.}
Assume $v\in V(J)\cap U_1^3$.  If $V(J)=\{v\}$, then we are done.  Otherwise,
let $w$ be a neighbor of $v$.  If $w$ is not a needy 3-neighbor of $v$, then
$\ch^*(v)\ge 5$, a contradiction.  Further, $v$ has at most one needy
3-neighbor.  Thus, we are done.

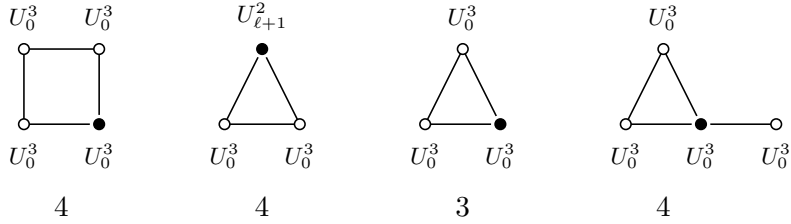
\begin{figure}[!htb]
\centering
\begin{tikzpicture}[semithick]
\tikzstyle{IStyle}=[shape = circle, minimum size = 6.5pt, draw=none, fill=white,
outer sep=0pt, inner sep=2.5pt]

\draw (0,0) node[u03Style] {} -- (1,0) node[u03Style] {} 
-- (1,-1) node[IStyle] {} node[u03LStyle,fill=black] {} 
-- (0,-1) node[u03LStyle] {} 
-- (0,0);
\draw (.5,-2.1) node[lStyle] {4};

\begin{scope}[xshift=1.05in]
\draw (.5,0) node[IStyle] {} node[u12Style,fill=black] {} 
-- (1,-1) node[u03LStyle] {}
-- (0,-1) node[u03LStyle] {}
-- (.5,0);
\draw (.5,-2.1) node[lStyle] {4};
\end{scope}

\begin{scope}[xshift=2.1in]
\draw (.5,0) node[u03Style] {} 
-- (1,-1) node[IStyle] {} node[u03LStyle,fill=black] {}
-- (0,-1) node[u03LStyle] {} 
-- (.5,0);
\draw (.5,-2.1) node[lStyle] {3};
\end{scope}

\begin{scope}[xshift=3.15in]
\draw (.5,0) node[u03Style] {} 
-- (1,-1) node[IStyle] {} node[u03LStyle,fill=black] {} 
-- (0,-1) node[u03LStyle] {} 
-- (.5,0) (1,-1) -- (2,-1) node[u03LStyle] {};
\draw (.5,-2.1) node[lStyle] {4};
\end{scope}
\end{tikzpicture}
\caption{The 4 possible components of $G\setminus U_{\ell}^2$ in Case 5, those
that have a cycle.\label{case5-fig}}
\end{figure}

\textbf{Case 5: $\boldsymbol{V(J)\subseteq U_{\ell+1}^2\cup U_0^3}$ and $J$ contains a
cycle.}
Let $C$ be a cycle in $J$; see Figure~\ref{case5-fig}.  It is easy to check that
each cycle vertex finishes with charge at least 1; thus $|C|\le 4$.  If $|C|=4$,
then each cycle vertex is in $U_0^3$ and has a neighbor in $U_{\ell}^2$.  Thus,
$J\cong C_4$.  Now suppose $|C|=3$.  If $C$ contains a vertex in $U_{\ell+1}^2$, then
it contains exactly one such vertex, and its other two vertices are in $U_0^3$,
each with a neighbor in $U_{\ell}^2$.  So $J\cong C_3$ (with a single vertex in
$U_{\ell+1}^2$).
So assume $C$ is a 3-cycle with all vertices in $U_0^3$.  If no vertex on $C$ has
a neighbor in $J\setminus C$, then we are done.  Otherwise, exactly one cycle
vertex does, and it is a needy 3-neighbor.

\textbf{Case 6: $\boldsymbol{V(J)\subseteq U_{\ell+1}^2\cup U_0^3}$ and $J$ is a tree.}
Let $T:=J$.  Let $n_2:=|U_{\ell+1}^2\cap V(T)|$ and $n_3:=|U_0^3\cap V(T)|$.
Recall that no vertex in $U_{\ell+1}^2$ has a neighbor in $U_{\ell}^2$, by
Lemma~\ref{no2-threads-k-odd}.  So each leaf of $T$ is in $U_0^3$.  Form $T'$
from $T$ by replacing paths with internal vertices in 
$U_{\ell+1}^2$ by edges. 
So $|V(T')|=n_3$.  Let $\ch^*(T'):=\ch^*(T)-2|U_{\ell+1}^2\cap V(T)|$.  Note that
$\ch^*(T')$ is precisely the sum of charges that would have ended on $T'$ if it
had appeared in $G$ when we did the discharging.  Since each vertex of $T'$ has
degree 3 in $G$, the number of edges (externally) incident to $T$ is
$3|T'|-\sum_{v\in T'}d_{T'}(v)=3|T'|-2(|T'|-1)=|T'|+2$.  Since $\ch(T')=3|T'|$, 
and $T'$ sends 2 along each incident edge,
we have $\ch^*(T')=3|T'|-2(|T'|+2)=|T'|-4=n_3-4$.  Since $\ch^*(T)\le 4$, we get that
$n_3\le 8$.  Recall that a vertex in $U_0^3$ with three neighbors in
$U_{\ell}^2$ is reducible, by Lemma~\ref{no3withall2s-lem}. So $n_3\ge 2$.
Note that $\ch^*(T)=n_3-4+2n_2\le 4$.  So $n_2 \le \frac{8-n_3}2$.
For brevity, we henceforth denote $|V(T')|$ by $|T'|$.  We consider the seven
possibilities when $|T'|\in\{2,\ldots,8\}$.

Suppose $|T'|=2$.  By Lemma~\ref{noadj3swith32nbrs-lem}, the edge of $T'$ must be
subdivided in $T$ by one or more vertices of $U_{\ell+1}^2$.  
We have $n_2\le 3$, which gives the 3 possibilities in Figure~\ref{case6T2-fig}.

\begin{figure}[!htb]
\centering
\begin{tikzpicture}[semithick]
\tikzstyle{IStyle}=[shape = circle, minimum size = 6.5pt, draw=none, fill=white,
outer sep=0pt, inner sep=2.5pt]
\draw (0,0) node[u03Style] {} -- (1,0) node[IStyle] {} node[u12Style, fill=black] {} -- (2,0) node[u03Style]
{};
\draw (1,-.6) node[lStyle] {0};

\begin{scope}[xshift=1.4in]
\draw (0,0) node[u03Style] {} -- (1,0) node[IStyle] {} node[u12Style, fill=black] {} -- (2,0) node[u12Style]
{} -- (3,0) node[u03Style] {};
\draw (1.5,-.6) node[lStyle] {2};
\end{scope}

\begin{scope}[xshift=3.2in]
\draw (0,0) node[u03Style] {} -- (1,0) node[IStyle] {} node[u12Style, fill=black] {} -- (2,0)
node[u12Style] {} -- (3,0) node[IStyle] {} node[u12Style, fill=black] {} -- (4,0) node[u03Style] {};
\draw (2,-.6) node[lStyle] {4};
\end{scope}

\end{tikzpicture}
\caption{The 3 possible components of $G\setminus U_0^2$ in Case 6 when
$|T'|=2$.\label{case6T2-fig}}
\end{figure}
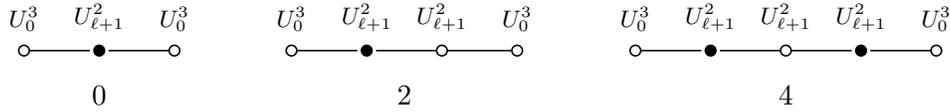

Suppose $|T'|=3$.  By Lemma~\ref{noadj3swith32nbrs-lem},  no vertex in $U_0^3$ has both
a needy 3-neighbor and a neighbor in $U_{\ell}^2$.  Thus, each edge of $T'$ must
be subdivided in $T$ by a vertex in $U_{\ell+1}^2$.  Recall that $n_2\le
\frac{8-n_3}2$.  So $n_3=3$ and $n_2=2$.

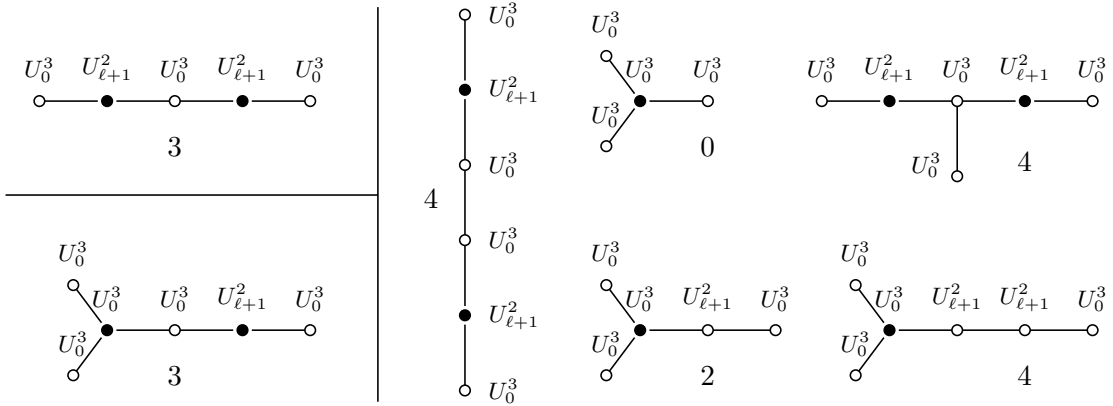
\begin{figure}[!t]
\centering
\begin{tikzpicture}[semithick, xscale=.9]
\tikzstyle{IStyle}=[shape = circle, minimum size = 6.5pt, draw=none, fill=white,
outer sep=0pt, inner sep=2.5pt]
\begin{scope}
\draw (0,0) node[u03Style] {} -- (1,0) node[IStyle] {} node[u12Style,fill=black] {} -- (2,0) node[u03Style
] {} -- (3,0) node[IStyle] {} node[u12Style,fill=black] {} -- (4,0) node[u03Style] {};
\draw (2,-.6) node[lStyle] {3};
\end{scope}

\draw[] (-.5,-1.25) -- (5.0,-1.25) (5.0,1.25) -- (5.0,-4.0);

\begin{scope}[yshift=-1.2in]
\draw 
(.5,.6) node[u03Style] {} -- (1,0) 
(.5,-.6) node[u03Style] {} -- (1,0) node[IStyle] {}  
node[u03Style,fill=black] {} -- (2,0) node[u03Style
] {} -- (3,0) node[IStyle] {} node[u12Style,fill=black] {} -- (4,0) node[u03Style] {};
\draw (2,-.6) node[lStyle] {3};
\end{scope}

\begin{scope}[xshift=1in]
\begin{scope}[xshift=-.1in]
\draw (4.0, -3.85) node[u03RbStyle] {} 
--++ (0,1) node[IStyle] {} node[u12RStyle,fill=black] {} 
--++ (0,1) node[u03RbStyle] {} 
--++ (0,1) node[u03RbStyle] {} 
--++ (0,1) node[IStyle] {} node[u12RStyle,fill=black] {} 
--++ (0,1) node[u03RbStyle] {};
\draw (3.5,-1.30) node[lStyle] {4};
\end{scope}

\begin{scope}[xshift=2.1in]
\draw 
(.5,.6) node[u03Style] {} -- (1,0) 
(.5,-.6) node[u03Style] {} -- (1,0) 
node[IStyle] {} node[u03Style,fill=black] {} -- (2,0) node[u03Style] {}; 
\draw (2,-.6) node[lStyle] {0};
\end{scope}

\begin{scope}[xshift=2.1in, yshift=-1.2in]
\draw 
(.5,.6) node[u03Style] {} -- (1,0) 
(.5,-.6) node[u03Style] {} -- (1,0) 
node[IStyle] {} node[u03Style,fill=black] {} -- (2,0) node[u12Style
] {}  -- (3,0) node[u03Style] {}; 
\draw (2,-.6) node[lStyle] {2};
\end{scope}

\begin{scope}[xshift=3.55in]
\draw (0,0) node[u03Style] {} -- (1,0) node[IStyle] {} node[u12Style,fill=black] {} -- (2,0) node[u03Style
] {} -- (3,0) node[IStyle] {} node[u12Style,fill=black] {} -- (4,0) node[u03Style] {} (2,0) --
(2,-1) node[uStyle,
label={[xshift=0cm, yshift=.15cm]left:\footnotesize{$U_0^3$}}] {};
\draw (3,-.8) node[lStyle] {4};
\end{scope}

\begin{scope}[xshift=3.55in, yshift=-1.2in]
\draw 
(.5,.6) node[u03Style] {} -- (1,0) 
(.5,-.6) node[u03Style] {} -- (1,0) 
node[IStyle] {} node[u03Style,fill=black] {} -- (2,0) node[u12Style
] {}  -- (3,0) node[u12Style] {}  -- (4,0) node[u03Style] {};

\draw (3,-.6) node[lStyle] {4};
\end{scope}
\end{scope}

\end{tikzpicture}
\caption{The 7 possible components of $G\setminus U_{\ell}^2$ in Case 6 when
$|T'|\in\{3,4,5\}$.\label{case6T345-fig}
}
\end{figure}

Suppose $|T'|=4$.
The only 4-vertex trees are $K_{1,3}$ and $P_4$.  
Recall that $n_2\le \frac{8-n_3}2=2$.  If $T'\cong P_4$, then $T$ must
contain a vertex in $U_{\ell+1}^2$ incident to each leaf.  There is a unique
such tree, a 6-vertex path with each neighbor of a leaf in $U_{\ell+1}^2$ (and
the four other vertices in $U_0^3$).  So assume $T'\cong K_{1,3}$.  
Now $n_2\in\{0,1,2\}$. This results in 1, 1, and 2
possibilities with orders 4, 5, and 6. 

Suppose $|T'|=5$.
The only 5-vertex subcubic trees are $P_5$ and $K_{1,3}$ with an edge subdivided.
Now $n_2\le 1$.  Thus, we cannot have $T'\cong P_5$, since then $T$ would have
a vertex in $U_0^3$ with both a needy 3-vertex and a neighbor in $U_{\ell}^2$,
which contradicts Lemma~\ref{noadj3swith32nbrs-lem}.
So $T'$ is formed from $K_{1,3}$ by subdividing a single edge.  Now we have a
single possibility for $T$, which is formed from $K_{1,3}$ by subdividing a
single edge twice.

Suppose $|T'|=6$.
Now $n_2\le 1$.  
There are 4 subcubic trees on 6 vertices.  However, two of them contain two
copies of a leaf adjacent to a vertex of degree 2 (in the tree).  Neither of
these are valid options for $T'$, by Lemma~\ref{noadj3swith32nbrs-lem}.  Thus, either
$T'$ is formed by subdividing a single edge of $K_{1,3}$ twice or else
$T'$ is a double-star (adjacent 3-vertices, with 4 leaves).  The first option
yields one case, and the second yields 3 cases (since we might not add a vertex
of $U_{\ell+1}^2$).

Suppose $|T'|=7$.
Now $n_2=0$; that is, $T'=T$.  Thus, each leaf of $T'$ must be adjacent to a
vertex of degree 3 in $T'$.  Since $T'$ has a 3-vertex, it has at least 3
leaves. Since each leaf has a neighbor of degree 3 in $T$, tree $T$ has at
least two 3-vertices.  There is a single possibility.

Suppose $|T'|=8$.
The analysis is nearly the same as when $|T'|=7$.  Now $T'$ must contain at
least 4 leaves and at least two 3-vertices.  Either $T'$ has 5 leaves and three
3-vertices or else $T'$ has 4 leaves, two 2-vertices, and two 3-vertices.
Each case gives a single possibility.
\end{proof}

\begin{figure}[!htb]
\centering
\begin{tikzpicture}[xscale=.8, semithick]
\tikzstyle{IStyle}=[shape = circle, minimum size = 6.5pt, draw=none, fill=white,
outer sep=0pt, inner sep=2.5pt]
 
\begin{scope}
\draw (0,0) node[u03LStyle] {} -- (1,0) node[IStyle] {} node[u03LStyle,fill=black] {} -- (2,0)
node[u03LStyle ] {} -- (3,0) node[IStyle] {} node[u03LStyle,fill=black] {} -- (4,0) node[u03LStyle] {}
(1,.75) node[u03Style] {} -- (1,0) (3,.75) node[u03Style] {} -- (3,0);
\draw (2,-1.2) node[lStyle] {3};
\end{scope}

\begin{scope}[xshift=2.1in]
\draw (0,0) node[u03LStyle] {} -- (1,0) node[IStyle] {} node[u03LStyle,fill=black] {} -- (2,0)
node[u03LStyle ] {} -- (3,0) node[IStyle] {} node[u03LStyle,fill=black] {} -- (4,0) node[u03LStyle] {}
(1,.75) node[u03Style] {} -- (1,0) (2,.75) node[u03Style] {} -- (2,0) (3,.75) node[u03Style] {} -- (3,0);
\draw (2,-1.2) node[lStyle] {4};
\end{scope}

\begin{scope}[xshift=4.2in]
\draw (0,0) node[u03LStyle] {} -- (1,0) node[IStyle] {} node[u03LStyle,fill=black] {} -- (2,0)
node[u03LStyle ] {} -- (3,0) node[u03LStyle] {} -- (4,0) node[IStyle] {} node[u03LStyle,fill=black] {}
-- (5,0) node[u03LStyle] {}
(1,.75) node[u03Style] {} -- (1,0) (4,.75) node[u03Style] {} -- (4,0);
\draw (2.5,-1.2) node[lStyle] {4};
\end{scope}

\draw[] (-.5,2.0) -- (16.25,2.0) (4.7,2) -- (4.7,-1.5);

\begin{scope}[xscale=.85]
\begin{scope}[xshift = -.2in, yshift=1.3in]
\draw 
(.5,.6) node[u03Style] {} -- (1,0) 
(.5,-.6) node[u03lStyle] {} -- (1,0) 
node[IStyle] {} node[u03rStyle,fill=black] {} -- (2,0) 
node[u03Style ] {} -- (3,0) 
node[u03Style ] {} -- (4,0) 
node[IStyle] {} node[u12Style,fill=black] {} -- (5,0) node[u03Style] {};
\draw (2,-.6) node[lStyle] {4};
\end{scope}

\begin{scope}[xshift=2.4in, yshift=1.3in]
\draw 
(.5,.6) node[u03Style] {} -- (1,0) 
(.5,-.6) node[u03lStyle] {} -- (1,0) 
node[IStyle] {} node[u03rStyle,fill=black] {} -- (2,0) node[u03lStyle
] {} -- (2.5,.6) node[u03Style] {} (2,0) -- (2.5,-.6) node[u03rStyle] {};
\draw (1.5,-.6) node[lStyle] {2};
\end{scope}

\begin{scope}[xshift=3.9in, yshift=1.3in]
\draw (.5,.6) node[u03Style] {} -- (1,0) 
(.5,-.6) node[u03lStyle] {} -- (1,0) 
node[IStyle] {} node[u03rStyle,fill=black] {} -- (2,0) node[u12Style] {} 
-- (3,0) node[IStyle] {} node[u03lStyle, fill=black] {}
-- (3.5,.6) node[u03Style] {} (3,0) -- (3.5,-.6) node[u03rStyle] {};
\draw (2,-.6) node[lStyle] {4};
\end{scope}

\begin{scope}[xshift=5.85in, yshift=1.3in]
\draw 
(.5,.6) node[u03Style] {} -- (1,0) 
(.5,-.6) node[u03lStyle] {} -- (1,0) 
node[IStyle] {} node[u03rStyle,fill=black] {} -- (2,0) node[u03lStyle
] {} -- (2.5,.6) node[u03Style] {} (2,0) -- (3,-.3) node[IStyle] {} node[u12Style, fill=black] {}
-- (4,-.6) node[u03Style] {};
\draw (1.5,-.6) node[lStyle] {4};
\end{scope}
\end{scope} 

\end{tikzpicture}
\caption{The 7 possible components of $G\setminus U_{\ell}^2$ in Case 6 when $|T'|\in
\{6,7,8\}$.\label{case6T678-fig}}
\end{figure}
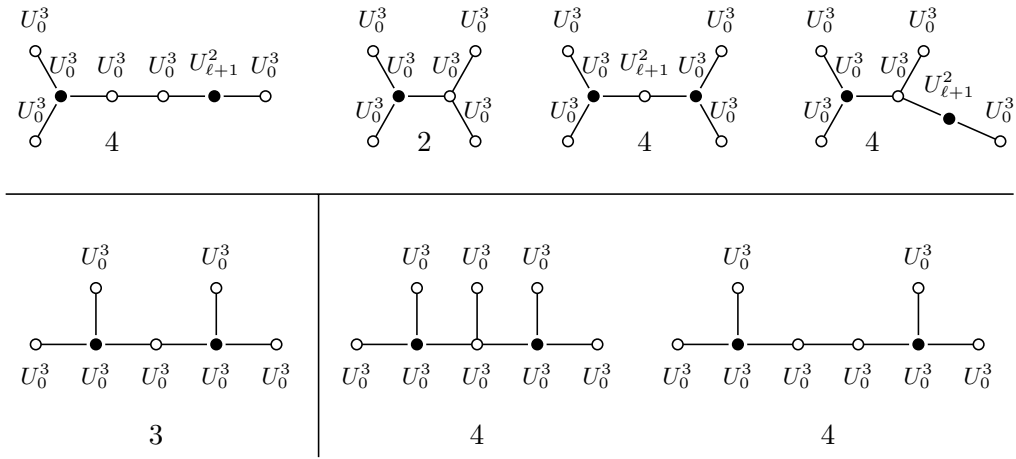

\begin{lem}
$G$ has an $(I,F_k)$-coloring, and is thus not a counterexample.
\end{lem}
\begin{proof}
We now construct an $(I,F_k)$-coloring of $G$.
As we described above, our plan is to color all vertices of $U_{\ell}^2$ with $I$
(since they form an independent set, by Lemma~\ref{no2-threads-k-odd}).
For each possible acyclic component $J$ of $G\setminus U_{\ell}^2$, shown in
Figures~\ref{cases1234-fig}, \ref{case6T2-fig}, \ref{case6T345-fig}, and
\ref{case6T678-fig}, we show how to extend this coloring to $J$.  Those vertices drawn
as white are colored with $F$ and those drawn as black are colored with $I$.
Doing this preserves that $I$ is an independent set and $G[F]$ is a forest with
at most $k$ vertices in each component.
The only complication is the four possible components $J$ that contain a cycle, shown
in Figure~\ref{case5-fig}.  In fact, the second and fourth of these are fine.  Suppose
instead that $J\in\{C_3,C_4\}$ with all vertices in $U_0^3$. Now we color one
vertex $v$ of $J$ with $I$ (and the rest with $F$).  To preserve that $I$ is an
independent set, we recolor the neighbor $w$ of $v$ in $U_{\ell}^2$ with $F$.  We
must ensure that $w$ does not become part of a tree on $k+1$ vertices.  Since
$\ch^*(J)\ge 3$, every other component $J'$ of $G\setminus U_{\ell}^2$ has
$\ch^*(J')\leq 1$; in particular, this is true of the component containing the
neighbor of $w$ other than $v$.  So $J'$ is either $K_{1,3}$ (with all vertices
in $U_0^3$) or else $P_3$ (with its center vertex in $U_{\ell+1}^2$ and leaves in
$U_0^3$).  In each case for $J'$, the subgraph induced by its vertices colored
$F$ is an independent set.  Thus, recoloring $w$ with $F$ creates a tree
colored $F$ with at most 2 vertices.
\end{proof}

\section*{Acknowledgments}
Thank you to two anonymous referees who both provided helpful feedback.  In
particular, one referee read the paper extremely carefully and caught numerous
typos, inconsistencies, and errors, and also suggested various improvements in the
presentation.

\bibliographystyle{siam}
{\footnotesize{\bibliography{refs}}
\end{document}